\documentclass[twoside]{amsart}
\usepackage{mathrsfs}
\usepackage{amsfonts}
\usepackage{amssymb}
\usepackage{amssymb}
\usepackage{amssymb}
\usepackage{amssymb}
\usepackage{amssymb}
\usepackage{amssymb}
\usepackage{amssymb}
\usepackage{mathdots}
\usepackage{amsmath}
\usepackage{amsthm}
\usepackage[all]{xy}

\usepackage{lastpage}

\RequirePackage{amsmath} \RequirePackage{amssymb}
\usepackage{amscd,latexsym,amsthm,amsfonts,amssymb,amsmath,amsxtra}
\usepackage
[
colorlinks=true, 
linkcolor=blue,
urlcolor=blue,
bookmarks=true,
bookmarksopen=true,
citecolor=blue,
]
{hyperref}

    \newcommand{\BA}{{\mathbb {A}}} 
    \newcommand{\BC}{{\mathbb {C}}}

     \newcommand{\sH}{{\mathscr {H}}}

    \newcommand{\CC}{{\mathcal {C}}}

    \newcommand{\CM}{{\mathcal {M}}}

    \newcommand{\CS}{{\mathcal {S}}} 
     
    \newcommand{\CW}{{\mathcal {W}}}

     \newcommand{\fo}{{\mathfrak{o}}}  \newcommand{\fp}{{\mathfrak{p}}}

    \newcommand{\RU}{{\mathrm {U}}}

    \newcommand{\diag}{{\mathrm {diag}}}\newcommand{\Mat}{{\mathrm {Mat}}}

    \newcommand{\lenth}{{\mathrm {\lenth}}}

     \newcommand{\GL}{{\mathrm{GL}}}
    \newcommand{\Hom}{{\mathrm{Hom}}} 
    \newcommand{\Ind}{{\mathrm{Ind}}}

    \renewcommand{\Re}{{\mathrm{Re}}}

     \newcommand{\rht}{{\mathrm{ht}}}

\renewcommand{\mod}{\ \mathrm{mod}\ }

\newcommand{\Supp}{\mathrm{Supp}}

 \newcommand{\SL}{{\mathrm{SL}}}

\newcommand{\vol}{{\mathrm{vol}}}  
 
\newcommand{\Char}{{\mathrm{Char}}}
 
 \newcommand{\Sp}{{\mathrm{Sp}}}

    \newcommand{\bx}{{\bf {x}}}   \newcommand{\bm}{{\bf {m}}}   \newcommand{\bn}{{\bf {n}}} \newcommand{\bt}{{\bf {t}}}
    \newcommand{\bW}{{\bf {W}}}

    \newcommand{\wt}{\widetilde}  
    
    \newcommand{\pair}[1]{\langle {#1} \rangle}
    \newcommand{\wpair}[1]{\left\{{#1}\right\}}

    \newcommand{\ov}{\overline}
    
    \newcommand{\incl}{\hookrightarrow}
    
     \newcommand{\ra}{\rightarrow}

    \theoremstyle{plain}
    
       \newtheorem*{theorem*}{Local Converse Theorem for $\Sp_{2r}$}

    \newtheorem{thm}{Theorem}[section] \newtheorem{cor}[thm]{Corollary}
    \newtheorem{lem}[thm]{Lemma}  \newtheorem{prop}[thm]{Proposition}

    \numberwithin{equation}{section}

\usepackage[top=1in, bottom=1in, left=1.25in, right=1.25in]{geometry}

\title{A local converse theorem for $\Sp_{2r}$}
\author{Qing Zhang}


\address{School of Mathematics, Sun Yat-Sen University, Guangzhou, China, 510275}
\email{zhang.qing@yahoo.com}

\begin{document}

\maketitle
\begin{abstract}
In this paper, we prove the local converse theorem for $\Sp_{2r}(F)$ over a $p$-adic field $F$. More precisely, given two irreducible supercuspidal representations of $\Sp_{2r}(F)$ with the same central character such that they are generic with the same additive character and they have the same gamma factors when twisted with generic irreducible representations of $\GL_n(F)$ for all $1\le n\le r$, then these two representations must be isomorphic. Our proof is based on the local analysis of the local integrals which define local gamma factors. A key ingredient of the proof is certain partial Bessel function property developed by Cogdell-Shahidi-Tsai recently. The same method can give the local converse theorem for $\RU(r,r)$. 
\end{abstract}

%\tableofcontents

\section{Introduction}
Let $F$ be a $p$-adic field and $\psi$ be a non-trivial additive character of $F$. Let $G=\Sp_{2r}(F)$ be the symplectic group of rank $r$ over $F$ and let $U$ be the maximal unipotent subgroup of a fixed Borel subgroup of $G$. We can define a generic character $\psi_U$ of $U$ from $\psi$.

The purpose of this paper is to prove the following 
\begin{theorem*}
 Let $\pi,\pi_0$ be two irreducible $\psi_U$-generic supercuspidal representations of $\Sp_{2r}(F)$ with the same central character. If $\gamma(s,\pi\times \tau,\psi)=\gamma(s,\pi_0\times \tau,\psi) $ for all irreducible generic representations $\tau$ of $\GL_k(F)$ and for all $k$ with $1\le k\le r$, then $\pi\cong \pi_0$.
\end{theorem*}

Here the local $\gamma$-factors are defined from the local functional equations of local zeta integrals which were considered in \cite{GePS87, GiRS97, GiRS98}. In \cite{Ka}, it is proved that these gamma factors agree with the local gamma factors arising from Langlands-Shahidi method. In particular, the local gamma factors satisfy multiplicativity. Thus in the conditions of the above theorem, one only needs to twist by irreducible generic supercuspidal representations of $\GL_k(F)$. 

With similar proof, the above local converse theorem also holds for quasi-split unitary group $\RU(r,r)$. Here the local gamma factors are defined from the local functional equations of local zeta integrals which were considered in \cite{BAS}.

The above theorem was conjectured by D. Jiang, see \cite{Jng, JngN}. To the author's understanding, by considering local descent map from $\GL_{2r+1}$ to $\Sp_{2r}$, Jiang and Soudry \cite[Theorem A7]{JngS} can reduce the above theorem to the Jacquet's local converse conjecture for $\GL_n$, which was recently proved by Chai \cite{Ch} and Jacquet-Liu \cite{JL} independently. But it seems that the proof of Jiang and Soudry is still not published. On the other hand, modulo the results of \cite{Ch, JL}, the above local converse theorem is equivalent to the irreducibility of the local descent from $\GL$ to $\Sp$ which is claimed in \cite[Theorem A7]{JngS}. Thus our proof of the above theorem can serve as an alternative proof of the irreducibility of the local descent map. In \cite{ST}, Soudry and Tanay considered the local descent from $\GL_{2r}$ to $\RU(r,r)$ and proved the descent is irreducible, which could reduce the above local converse theorem to the $\GL$ case. The details of a proof of the local converse theorem for $\RU(r,r)$ using descent recently appeared in \cite{M}, where the local gamma factors are Shahidi local gamma factors.

In any case, it seems valuable to give a proof of the local converse theorem for $\Sp_{2r}$ and $\RU(r,r)$ within the context of these groups themselves, i.e., without utilizing the local converse theorem for $\GL$. Furthermore, the method itself of our proof seems interesting and should be applicable to more local problems. 

We now briefly discuss the method we used. Certain local converse theorems for small symplectic and unitary groups were considered in \cite{ChZ, Zh1,Zh2,Zh3}, where the main tools are Howe vectors developed by Baruch \cite{Ba95,Ba97}. The difficulty to generalize the proof to higher rank groups is certain ``stability" property of Bessel functions, see Lemma \ref{lem4.3} (3) and the remark after it. Recently, Cogdell-Shahidi-Tsai \cite{CST17} developed a theory of partial Bessel functions which enabled them to prove stability of the exterior square local gamma factors for $\GL_n$. It turns out the method used in \cite{CST17} is general enough such that it can be applied to our case, which can overcome the difficulty mentioned above. As mentioned above, we believe the method used here might have more local applications. 

This paper is organized as follows. In section \ref{sec3}, we review the definitions of the local gamma factors. In section \ref{sec4}, we review the main tools we will use later: Howe vectors and the theory of Cogdell-Shahidi-Tsai on partial Bessel functions. In section \ref{sec5}, we prepare some materials which will be used in the proof of the main theorem. The proofs of the local converse theorem are given in sections \ref{sec6},\ref{sec7}. In the final section \ref{sec8}, we briefly discuss the local converse theorem for $\RU(r,r)$.

\section*{Acknowledgment }

I would like to thank my advisor Jim Cogdell for many useful discussions, constant support and encouragement over the years. I also would like to thank Friedrich Kopp for answering me questions on mathoverflow.net, which gave the proof of Proposition \ref{prop2.5} (2). I appreciate Baiying Liu for pointing to me the reference \cite{JngS}, and Dani Szpruch for useful communications on metaplectic groups. I thank Dihua Jiang, David Soudry, and Freydoon Shahidi for their interest and encouragement on this work. I am very grateful to the anonymous referee for his/her careful reading of this article and many useful suggestions which helped to improve this paper both linguistically and mathematically. 

\section{Notations}
Let $F$ be a $p$-adic field, i.e., a local field of characteristic 0. Denote by $\fo, \fp,\varpi$ the ring of integers of $F$, the maximal ideal of $\fo$, and a fixed generator of $\fp$ respectively.\\

Given two positive integers $n,m$, denote by $I_n$ the identity matrix of rank $n$, and by $\Mat_{m\times n}(F)$ the ring of $m\times n$ matrices with coefficients in $F$.\\

Let $r$ be a positive integer, $J_r=\begin{pmatrix} &&&1\\&&1&\\ &\iddots&&\\ 1&&&\end{pmatrix}$, and $w^{r}_{\ell}=\begin{pmatrix}&J_r\\ -J_r& \end{pmatrix}$. Let $F^{2r}$ be the space of row vectors of dimension $2r$. We endow $F^{2r}$ the symplectic form $\pair{~,~}$ defined by $w_{\ell}^r$, i.e., 
$$\pair{v_1,v_2}=2v_1\cdot w_{\ell}^r \cdot {}^t\!v_2, \quad v_1,v_2\in F^{2r}. \quad  \footnote{ Unlike the literature \cite{GiRS97, GiRS98, Ka}, we add an extra $2$ in front of the symplectic form, which will be used to simplify certain Weil representation formulas.}$$
The group $\Sp_{2r}(F)$ is defined to be the isometry group of $\pair{~,~}$, i.e., 
$$\Sp_{2r}(F)=\wpair{g\in \GL_{2r}(F): g\cdot w^{r}_{\ell}\cdot {}^t\! g=w^{r}_{\ell}}.$$ 
For $a\in \GL_r(F)$ and $b\in \Mat_{r\times r}(F)$ with $ bJ_r=J_r {}^t\!b $, we denote
$$ \bm_r(a)=\begin{pmatrix} a&\\ &a^*\end{pmatrix}\in \Sp_{2r}(F), \bn_r(b)=\begin{pmatrix}I_r& b\\ &I_r \end{pmatrix}, $$
where $a^*=J_r{}^t\!a^{-1} J_r$.

Let $M_r=\wpair{\bm_r(a),a\in \GL_r(F)}, N_r=\wpair{\bn_r(b): b\in \Mat_{r\times r}(F), J_r{}^t\! b =bJ_r},$ and $P_r=M_rN_r$, which is the Siegel parabolic subgroup of $\Sp_{2r}(F)$.

For $n<r$, denote $w_{r-n,n}=\begin{pmatrix}&I_{r-n} &&\\ I_n &&&\\ &&&I_n\\ &&I_{r-n}& \end{pmatrix}$, and $w_n=\begin{pmatrix} &I_n\\ -I_n \end{pmatrix}\in \Sp_{2n}(F)$. For $n<r$, we embed $\Sp_{2n}(F)$ into $\Sp_{2r}(F)$ by $g\mapsto \diag(I_{r-n}, g, I_{r-n})$. We will identify elements $g\in \Sp_{2n}(F)$ with elements of $\Sp_{2r}(F)$ under this embedding. By this convention, we have 
$$\bm_{n}(a)=\bm_{r}\begin{pmatrix}I_{r-n} &\\ &a \end{pmatrix}, a\in \GL_{n}(F).$$

We denote
 $$\tilde w^r_n=w_{r-n,n}^{-1} w_n w_{r-n,n}=\begin{pmatrix} &&I_n\\ &I_{2(r-n)} & \\ -I_n&& \end{pmatrix}.$$
 
 Let $Q_n^r=L_n^r V_n^r$ be the parabolic subgroup of $\Sp_{2r}(F)$ with Levi subgroup
 $$L_n^r=\wpair{\bm_r(\diag(a,a_{n+1},\dots,a_r)), a\in \GL_n(F), a_i\in \GL_1(F),n+1\le i\le r}.$$
 Note that $w_{r-n,n}^{-1}M_nw_{r-n,n}\subset L_n^r, w_{r-n,n}^{-1}P_n w_{r-n,n}\subset Q_n^r$. In fact, we have 
 $$w_{n-r,n}^{-1}\bm_n(a)w_{n-r,n}=\bm_r(\diag(a,1,\dots,1))\in L_n^r.$$
 For $a\in \GL_n(F)$, we will write $\bt_n(a)=\bm_r(\diag(a,1,\dots,1))\in L_n^r$.
 
 Let $B^r=A^rU^r$ be the upper triangular Borel subgroup of $\Sp_{2r}(F)$ with maximal torus $$A^r=\wpair{\diag(a_1,a_2,\dots,a_r,a_r^{-1},\dots,a_1^{-1}),a_i\in F^\times, 1\le i\le r},$$ and maximal unipotent $U^r$. We will also write $U^r$ as $U_{\Sp_{2r}}$ when we want to emphasize its dependence on the group $\Sp_{2r}$. For an integer $i$ with $1\le i\le r$, let $\alpha_i$ be the simple root defined by 
 $$\alpha_i( \diag(a_1,a_2,\dots,a_r,a_r^{-1},\dots,a_1^{-1}))=a_i/a_{i+1}, 1\le i\le r-1,$$
 and 
 $$\alpha_r( \diag(a_1,a_2,\dots,a_r,a_r^{-1},\dots,a_1^{-1}))=a_r^2. $$
 Let $\Delta^r=\wpair{\alpha_i,1\le i\le r}$ be the set of simple roots.  Any root $\beta$ of $\Sp_{2r}(F)$ can be uniquely written as $\beta=\sum_{i=1}^r c_i\alpha_i,$ with $c_i\in \wpair{0,\pm1,\pm2}$. The height of $\beta$ is defined to be $\rht(\beta)=\sum_i c_i$.

 For a root $\beta$ of $\Sp_{2r}(F)$, let $U^r_\beta$ be the root space of $\beta$ and $\bx_\beta:F\ra U_\beta$ be a fixed isomorphism.

\section{Review of the definition of the local gamma factors}\label{sec3}
In this section, we will give a review of definitions of local gamma factors for $\Sp_{2r}\times \GL_n$ following \cite{Ka}.

\subsection{The metaplectic group $\wt{\Sp}_{2n}(F)$}
Let $\wt{\Sp}_{2n}(F)$ be the metaplectic double cover of $\Sp_{2n}(F)$ defined by the Rao cocycle \cite{Rao93}, which is a map $c: \Sp_{2n}(F)\times \Sp_{2n}(F)\ra \wpair{\pm 1}$. The group $\wt{\Sp}_{2n}(F)$ is then realized as $\Sp_{2n}(F)\times \wpair{\pm 1}$ with multiplication 
$$(g_1,\epsilon_1)(g_2,\epsilon_2)=(g_1g_2, \epsilon_1 \epsilon_2 c(g_1,g_2)).$$
We state here certain formulas we need for Rao cocycles. Let $\bx: \Sp_{2n}(F)\ra F^\times/F^{\times,2}$ be the map defined in \cite[Lemma 5.1]{Rao93}. The Rao cocycle is defined in \cite[Theorem 5.3]{Rao93}. 
%We will use the following formula
%$$c(\bm_n(a) \bn_n(b), g)=(\det(a), \bx(g)), a\in \GL_n(F), b\in \Mat_{n\times n}(F), J_n{}^t\! b=bJ_n, g\in \Sp_{2n}(F). $$
%This formula can be found in \cite[p.455]{Sz}. 

 A typical element $\tilde g\in \wt{\Sp}_{2n}(F)$ is of the form $(g,\epsilon), g\in \Sp_{2n}(F), \epsilon\in \wpair{\pm 1}$. An element $g\in \Sp_{2n}(F)$ is identified with $(g,1)\in \wt{\Sp}_{2n}(F)$ (the map $g\mapsto (g,1)$ is not a group homomorphism). The map $(g,\epsilon)\mapsto g$ defines the double cover projection map $\wt{\Sp}_{2n}(F)\ra \Sp_{2n}(F)$. For a subset $S\subset \Sp_{2n}(F)$, we denote by $\widetilde S$ the inverse image of $S$ under the projection $\wt{\Sp}_{2n}(F)\ra \Sp_{2n}(F)$.

Let $\psi$ be a non-trivial additive character of $F$. For $a\in F^\times$, let $\psi_a$ be the character of $F$ defined by $\psi_a(x)=\psi(ax)$.  Let $\gamma(\psi)$ be the Weil index of $x\mapsto \psi(x^2)$, and let
$$\gamma_{\psi}(a)=\frac{\gamma(\psi_a)}{\gamma(\psi)}, a\in F^\times,$$
see \cite[Appendix]{Rao93}.  We have the property $\gamma_{\psi^{-1}}(x)=\gamma_{\psi}^{-1}(x),x\in F^\times$. 

Let $\tilde F^\times$ be the double cover of $F^\times$ defined by the Hilbert symbol, i.e., $\tilde F^\times =F^\times \times \wpair{\pm 1}$ with multiplication $(a_1,\epsilon_1)\cdot (a_2, \epsilon_2)=(a_1a_2, \epsilon_1 \epsilon_2 (a_1,a_2)_F), a_1,a_2\in F^\times, \epsilon_1,\epsilon_2\in F^\times$. Here $(~,~)_F$ is the Hilbert symbol of $F$. 

The function $\gamma_\psi$ satisfies the property $\gamma_\psi(ab)=\gamma_\psi(a)\gamma_\psi(b)(a,b)_F$, see \cite[Theorem A.4]{Rao93}. Thus $\gamma_\psi$ extends to a character of $\tilde F^\times$ by $\gamma_\psi(a,\epsilon)=\epsilon \gamma_\psi(a)$.

\subsection{Weil representations of $\wt{\Sp}_{2n}(F)$}
Let $F^{2n}$ be the space of $2n$-dimensional row vectors. Recall that $F^{2n}$ is endowed with a symplectic form $\pair{~,~}$ defined by $\pair{v_1,v_2}=2v_1\cdot w_{\ell}^n \cdot {}^t\! v_2$. 
%Note that, in \cite{GiRS97, GiRS98, Ka}, the symplectic form is defined without the extra 2. We add this extra 2 to simplify certain Weil representation formula. 

 Let $ \sH_n=F^{2n}\oplus F$ be the $(2n+1)$-dimension Heisenberg group. The group operation is defined by 
$$[v_1,t_1]\cdot [v_2, t_2]=\left[v_1+v_2, t_1+t_2+\frac{1}{2}\pair{v_1,v_2} \right], v_1,v_2\in F^{2n}, t_1,t_2\in F.$$
The group $\sH_n$ can be embedded into $ U_{\Sp_{2n+2}}$ by 
$$[(x,y),t ]\mapsto \begin{pmatrix}1&x&y &t\\ &I_n&&J_n{}^t\! y \\ &&I_n&-J_n{}^t\! x\\ &&&1 \end{pmatrix}, (x,y\in F^n, t\in F).$$
Denote $X_n=\wpair{[(x,0),0]: x\in F^n}$ and $Y_n=\wpair{[(0,y),0], y\in F^n}$. For $n<r$, we will identify $X_n, Y_n$ with a subgroup of $\Sp_{2r}$ under the above identification and the embedding $\Sp_{2n+2}\ra \Sp_{2r}$. 

Let $\psi$ be a nontrivial additive character $F$. Let $\omega_\psi$ be the Weil representation of $\sH_n\rtimes \wt{\Sp}_{2n}$ realized on $\CS(F^n)$, the Bruhat-Schwarts functions on the row space $F^n$. We have the following formulas:
\begin{align*}
\omega_\psi([(x,0),z])\phi(\xi)&=\psi(z)\phi(\xi+x), \\
\omega_\psi([(0,y),0])\phi(\xi)&=\psi(2\xi J_n {}^t\! y)\phi(\xi),\\
\omega_\psi((\bm_n(a),\epsilon))&=\epsilon \gamma_\psi(\det(a))|\det(a)|^{1/2}\phi(\xi a),\\
\omega_\psi((\bn_n(b),\epsilon))\phi(\xi)&=\epsilon\psi(\xi J_n {}^t\! b {}^t\! \xi)\phi(\xi),\\
\omega_\psi(w_n)\phi(\xi)&=\beta_\psi \int_{F^n} \phi(x)\psi(2x J_n {}^t\! \xi)dx.
\end{align*}
Here $\phi\in \CS(F^r)$, and $\beta_\psi$ is certain fixed eighth root of unity. Recall that $w_n=\begin{pmatrix} &I_n\\ -I_n& \end{pmatrix}$. These formulas look a little bit different from that in \cite{GiRS98, Ka} since we used a little bit different symplectic form. For these formulas, see \cite{Ku}.

\subsection{Genuine induced representation of $\wt{\Sp}_{2n}(F)$}
Recall that $P_n=M_nN_n$ is the Siegel parabolic subgroup of $\Sp_{2n}(F)$ and $\wt{M}_n$ is the inverse image of $M_n$ in $\wt{\Sp}_{2n}(F)$. Then $\wt{M}_n$ is the double cover of $M_n\cong \GL_n(F)$ defined by the Hilbert symbol $(~,~)_F$. Let $\tau$ be an irreducible representation of $\GL_n(F)\cong M_n$, and $\psi$ be a nontrivial additive character of $F$. Let $s\in \BC$, and let $\tau_s\otimes \gamma_\psi^{-1}$ be the genuine representation of $\wt{M}_n$ defined by 
$$\tau_s\otimes \gamma_\psi^{-1}((a,\epsilon))=\epsilon \gamma_\psi^{-1}(\det(a))|\det(a)|^s\tau(a), a\in \GL_n(F), \epsilon \in \wpair{\pm 1}.$$
For $s\in \BC$, we consider the induced representation 
$$\tilde I(s,\tau,\psi)=\Ind_{\wt{P_n}}^{\wt{\Sp}_{2n}(F)}(\tau_{s-1/2}\otimes \gamma_\psi^{-1}) .$$
A typical element $f_s\in \tilde I(s,\tau,\psi)$ is a smooth function from $\wt{\Sp}_{2n}(F)$ to the space of $\tau$, which satisfies the relation 
$$f_s((\bm_n(a), \epsilon) u \tilde g)=\epsilon \delta_{P_n}^{1/2}(a) |\det(a)|^{s-1/2}\gamma_\psi^{-1}(\det(a)) \tau(a)f_s(\tilde g), $$
for $a\in \GL_n(F), \epsilon \in \wpair{\pm 1},u\in N_n, \tilde g\in \wt{\Sp}_{2n}(F)$. Here $\delta_{P_n}$ is the modulus character of $P_n$. 

Let $\psi_{U_{\GL_n}}$ be the generic character of the standard upper triangular unipotent subgroup $U_{\GL_n}$ of $\GL_n(F)$ defined by 
$$\psi_{U_{\GL_n}}( u)=\psi\left(\sum_{i=1}^{n-1} u_{i,i+1}\right). $$
Assume that $\tau$ is a generic irreducible representation of $\GL_n(F)$. We fix a nonzero Whittaker functional 
$$\lambda\in \Hom_{U_{\GL_n}}(\tau,\psi_{U_{\GL_n}}).  $$
Given an element $f_s\in \tilde I(s,\tau,\psi)$, we consider the $\BC$-valued function on $\wt{\Sp}_{2n}(F)\times \GL_n(F)$:
$$\xi_{f_s}(\tilde g,a)=\lambda(\tau(a)f_s(\tilde g)).$$
From the quasi-invariance of $f_s$, we have the following relation
$$\xi_{f_s}( (\bm_n(u),\epsilon) u' \tilde g,I_n)= \epsilon\psi_{U_{\GL_n}}(u)\xi_{f_s}(\tilde g,I_n),$$
for $u\in U_{\GL_n}, u'\in N_n, \epsilon \in \wpair{\pm 1}$. We denote 
$$\wt{V}(s,\tau,\psi)=\wpair{\xi_{f_s}, f_s\in \tilde I(s,\tau,\psi)}.$$
Here we remark that in the notation $\wt{V}(s,\tau,\psi)$, there are two places depending on $\psi$. The first one is $\gamma_\psi$ and the second one is the fixed Whittaker functional $\lambda\in \Hom_{U_{\GL_n}}(\tau,\psi_{U_{\GL_n}})$. If we replace $\psi$ by a different $\psi'$, one should change $\psi$ to $\psi'$ in both places. \\

Let $\tau^*$ denote the representation of $\GL_n(F)$ defined by $\tau^*(a)=\tau(a^*)$, where $a^*=J_n {}^t\! a^{-1}J_n\in \GL_n(F)$. Note that $\tau^*$ is isomorphic to the contragredient representation of $\tau$.

There is a (standard) intertwining operator $M(s,\tau,\psi): \wt{V}(s,\tau,\psi)\ra \wt{V}(1-s,\tau^*, \psi)$ defined by 
$$M(s,\tau,\psi)\xi_s(\tilde g, a)=\int_{N_n}\xi_s(w_n^{-1}u \tilde g, d_na^*)du,$$
where $d_n=\diag(-1,1,\dots,(-1)^n)\in \GL_n(F)$. It is a standard fact that this intertwining operator $M(s,\tau,\psi)$ is well-defined for $\Re(s)>>0$ and has a meromorphic continuation. 

\subsection{The local zeta integral and local gamma factors}
Let $\psi$ be a nontrivial additive character of $F$. Let $U=U^r$ be the unipotent subgroup of the upper triangular Borel subgroup of $\Sp_{2r}(F)$. We define a generic character $\psi_U$ of $U$ by 
$$\psi_U(u)=\psi\left(\sum_{i=1}^r u_{i,i+1}\right), u=(u_{i,j})\in U.$$

Let $\pi$ be an irreducible $\psi_U$-generic representation of $\Sp_{2r}(F)$ and let $\CW(\pi,\psi_U)$ be its $\psi_U$-Whittaker model. Let $n$ be a positive integer with $1\le n\le r$ and $\phi\in \CS(F^n)$. Let $\tau$ be an irreducible generic representation of $\GL_n(F)$ and $\xi_s\in \wt{V}(s,\tau,\psi^{-1})$. Then for $x\in X_n$, the function on $\wt{\Sp}_{2n}(F)$
$$\tilde g\mapsto \omega_{\psi^{-1}}(\tilde g)\phi(x) \xi_s(\tilde g, I_n)$$
descends to a function on $\Sp_{2n}(F)$. For $W\in \CW(\pi,\psi_U)$, the Shimura type integral for $\pi\times \tau$ is defined by 
\begin{align}\label{eq1.1}
&\Psi(W,\phi,\xi_s)\\
=&\left\{\begin{array}{lll}\int_{U^n\setminus \Sp_{2n}}\int_{R^{r,n}}\int_{X_n} W(w_{r-n,n}^{-1}(rxg)w_{r-n,n})\omega_{\psi^{-1}}(g)\phi(x)\xi_s(g,I_n)dxdrdg, & n<r,\\
                                                              \int_{U^r\setminus \Sp_{2r}} W(g)\omega_{\psi^{-1}}(g)\phi(e_r)\xi_s(g,I_r)dg, & n=r.        \end{array}\right.\nonumber \end{align}

Here $$R^{r,n}=\wpair{\bm_r \begin{pmatrix} I_{r-n-1} & &y\\ &1&\\ &&I_n \end{pmatrix} \in \Sp_{2r}(F)},$$
and $e_r=(0,\dots,0,1)\in F^r$. We apologize to use $r$ twice: one for the rank in $\Sp_{2r}$, the other one for element in $R^{r,n}$. Hopefully the meaning of $r$ is clear from the context. \\
\noindent \textbf{Remark:} These Shimura type integrals were first considered in \cite{GeJ78} when $n=r=1$, and then in \cite{GePS87} when $n=r\ge 1$. The most general cases were considered in \cite{GiRS97} (when $n=1$) and \cite{GiRS98} (for general $n,r$). Note that when $r<n$, similar Shimura type integrals can be defined. Since we are not going to use those cases, we did not include these integrals here. \\

The basic properties of these integrals about convergence, non-vanishing and meromorphic continuation were dealt with in \cite{GiRS97, GiRS98}. 

 \begin{thm}
There exists a meromorphic function $\gamma(s,\pi\times \tau,\psi)$ such that 
$$\Psi(W,\phi, M(s,\tau,\psi)\xi_s)=\gamma(s,\pi\times \tau,\psi)\Psi(W,\phi, \xi_s),$$
for all $W\in \CW(\pi,\psi_U), \phi \in \CS(F^n), \xi_s\in \wt{V}(s,\tau,\psi)$. 
\end{thm}
The proof of the existence of the local gamma factors depends on the uniqueness of Fourier-Jacobi models for $\Sp_{2n}(F)$, which was recently proved in \cite{GaGP, Su}. Some more details can be found in \cite[$\S$3,4]{Ka}.\\
\noindent \textbf{Remark:} Unlike in the \cite{Ka} case, we did not normalize our intertwining operator. Let $\Gamma(s,\pi\times \tau,\psi)$ be the normalized local gamma factor which is defined on \cite[p.408]{Ka}. It is known that $\Gamma(s,\pi\times \tau,\psi)$ and $\gamma(s,\pi\times \tau,\psi)$ differ by a factor which only depends on $\tau$ and the central character of $\pi$. Thus if the local converse theorem is true if one uses $\gamma(s,\pi\times \tau,\psi)$, then it is also true if one uses $\Gamma(s,\pi\times \tau,\psi)$ in the statement. On the other hand, it is shown in \cite{Ka} that the gamma factors $\Gamma(s,\pi\times \tau,\psi)$ have multiplicativity properties. Thus one only has to use the twists by supercuspidal representations for $\GL_n(F)$ in the local converse theorem.

\section{Howe vectors and partial Bessel functions}\label{sec4}
Fix a positive integer $r$ and denote $G=\Sp_{2r}(F)$. We will ignore the sup-script $r$ from various notations. For example we will write $\tilde w_n^r$ as $\tilde w_n$.

 Let $Z$ be the center of $G$. Note that $Z=\wpair{\pm I_{2r}}.$ Let $\omega$ be a character of $Z$ and let $C_c^\infty(G,\omega)$ be the space of compactly supported smooth functions $f$ on $G$ such that $f(zg)=\omega(z)f(g), z\in Z,g\in G$. Note that if $\pi$ is an irreducible super-cuspidal representation of $G$ with central character $\omega$, then the space $\CM(\pi)$ of matrix coefficients of $\pi$ is a subspace of $C_c^\infty(G,\omega)$.

Recall that $U$ is the maximal unipotent subgroup of the Borel subgroup $B$. Let $\psi$ be an unramified non-trivial additive character of $F$ and let $\psi_U$ be the corresponding generic character of $U$. Denote $C^\infty(G,\psi_U,\omega)$ the space of functions $W$ on $G$ such that $W(zg)=\omega(z)W(g)$, $W(ug)=\psi_U(u)W(g)$ for all $z\in Z,u\in U,g\in G$, and there exists an open compact subgroup $K$ of $G$ such that $W(gk)=W(g)$ for all $g\in G, k\in K$. Note that if $\pi$ is a $\psi_U$-generic irreducible representation of $G$ with central character $\omega$, then its $\psi_U$-Whittaker model $\CW(\pi,\psi_U)$  is a subspace of $C^\infty(G,\psi_U,\omega)$.

\subsection{Howe vectors}

In this subsection, we give a review on Howe vectors following \cite{Ba95}. The proofs can be found in \cite{Ba95}.

Let $m>0$ be a positive integer and $K_m^{\Sp_{2r}}=(I_{2r}+\textrm{Mat}_{2r\times 2r}(\fp^m))\cap \Sp_{2r}(F)$ be the standard congruence subgroup of $\Sp_{2n}(F)$. Let $\psi$ be a fixed additive character of $F$ with conductor $\fo$. Consider the character $\tau_m$ of $K_m$ defined by
$$\tau_m((k_{ij}))=\psi\left(\varpi^{-2m} (\sum_{i=1}^{r} k_{i,i+1})\right).$$
It is easy to check that $\tau_m$ is indeed a character. Let 
$$e_m=\diag(\varpi^{-m(2r-1)}, \varpi^{-m(2r-3)},\dots, \varpi^{-m}, \varpi^m, \dots, \varpi^{m(2r-1)})\in \Sp_{2r}(F)$$
and $H^r_m=e_mK^{\Sp_{2r}}_m e_m^{-1}$. Define a character $\psi_m$ on $H^r_m$ by $\psi_m(h)=\tau_m(e_m^{-1}h e
_m), h\in H^r_m$. If the group $\Sp_{2r}$ is fixed, we will write $H_m$ for $H_m^r$. Denote $U^r_m=U^r\cap H^r_m$. Note that we have $U^r=\cup_{m\ge 1}U^r_m$. 

In matrix form we have 
$$H_m=\begin{pmatrix}1+\fp^m & \fp^{-m} &\fp^{-3m} & \fp^{-5m}& \dots & \fp^{-(4r-3)m}\\ 
                                      \fp^{3m}& 1+\fp^m & \fp^{-m} & \fp^{-3m}& \dots & \fp^{-(4r-5)m} \\
                                      \fp^{5m} & \fp^{3m} & 1+\fp^m &\fp^{-m}& \dots & \fp^{-(4r-7)m}\\
                                      \fp^{7m} &\fp^{5m} & \fp^{3m} &1+\fp^m &\dots&\fp^{-(4r-9)m}\\
                                      \dots & \dots & \dots & \dots &\dots& \dots \\
                                      \fp^{(4r-1)m} & \fp^{(4r-3)m} & \fp^{(4r-5)m}&\fp^{(4r-7)m} & \dots & 1+\fp^m
                              \end{pmatrix}\cap \Sp_{2r}(F).$$

Recall that, for a root $\gamma$, $U_\gamma$ denotes the corresponding root space, and $\rht(\gamma)$ denotes the height of $\gamma$. Denote $U_{\gamma,m}=U_{\gamma}\cap H_m$ and denote by $\Sigma^+$ the set of all positive roots of $G$.
\begin{lem}\label{lem2.1}
\begin{enumerate}
\item The two characters $\psi_U$ and $\psi_m$ agree on $U_m=U\cap H_m$. 
\item For a positive root $\gamma$ of $\Sp_{2n}$, then $$U_{\gamma,m}=\wpair{\bx_{\gamma}(x): x\in \fp^{-(2\rht(\gamma)-1)m}},$$
and $$U_{-\gamma,m}=\wpair{\bx_{-\gamma}(x): x\in \fp^{(2\rht(\gamma)+1)m }}.$$
Moreover, we have $$U_m=\prod_{\gamma\in \Sigma^+}U_{\gamma,m},$$
where the product on the right side can be taken in any fixed order of $\Sigma^+$.
\end{enumerate}
\end{lem}

Let $\omega$ be a character of $Z$ and we have defined the space $C^\infty(G,\psi_U,\omega)$. Given a function $W\in C^\infty(G,\psi_U,\omega)$ such that $W(1)=1$, and a positive integer $m>0$, we consider the function $W_m$ on $G$ defined by
$$W_m(g)=\frac{1}{\vol(U_m)}\int_{U_m}\psi_m(u)^{-1}W(gu)du.$$
Let $C=C(W)$ be an integer such that $W$ is fixed by $K_C^{\Sp_{2r}}$ on the right side, then a function $W_m$ with $m\ge C$ is called a \textbf{Howe vector} following \cite{Ba95, Ba97}.

\begin{lem}\label{lemma22}
We have
\begin{enumerate}
\item $W_{m}(1)=1;$
\item if $m\ge C$, then $W_m(gh)=\psi_m(h)W_m(g)$, for all $h\in H_m.$
\end{enumerate}
\end{lem}
The proof of the above lemma can be found in \cite[Lemma 3.2]{Ba95}.

By (2) of Lemma \ref{lemma22}, for $m\ge C$, the function $W_{m}(g)$ satisfies the relation
\begin{equation}\label{eq2.1}W_{m}(u gh)=\psi_U(u)\psi_m(h)W_{m}(g), \forall u\in U, h\in H_m, g\in \Sp_{2n}(F). \end{equation}
Due to this relation, we also call $W_{m}$ a partial Bessel function.\\

Given $f\in C_c^\infty(G,\omega)$, we consider 
$$W^f(g)=\int_U \psi_U^{-1}(u)f(ug)du.$$
Note that the above integral is well-defined since $Ug$ is closed in $G$ and $f$ has compact support in $G$. Since $f$ is locally constant and has compact support, one can find a positive integer $C=C(f)$ such that $W^f(gk)=W^f(g)$ for all $g\in G, k\in K_C$. Thus $W^f\in C^\infty(G,\psi_U,\omega)$. We take a function $f\in C_c^\infty(G,\omega)$ such that $W^f(1)=1$. Given a positive integer $m>C(f)$, we consider the corresponding partial Bessel function 
\begin{equation}\label{eq2.2}B_m(g,f):=(W^f)_m(g)=\frac{1}{\vol(U_m)}\int_{U\times U_m}\psi_U^{-1}(u)\psi_m^{-1}(u')f(ugu')dudu'.\end{equation}

\subsection{Bruhat order} 
Let $\bW=\bW_G$ denote the Weyl group of $G=\Sp_{2r}(F)$. Denote by $e$ the unit element in $\bW$. For $w\in \bW$, denote $C(w)=BwB$. The Bruhat order on $\bW$ is defined by $w\le w'$ iff $C(w)\subset \overline{C(w')}$. We have $$\ov{C(w')}=\coprod_{w: w\le w'}C(w).$$

For $w\in \bW$, we denote $\Omega_{w}=\coprod_{w'\ge w}C(w')$. For any $w\in \bW$, $C(w)$ is closed in $\Omega_w$. 
\begin{prop}\label{prop2.5}
\begin{enumerate}
\item If $w,w'\in \bW$ with $w'>w$, then $\Omega_{w'}$ is an open subset of $\Omega_w$.  In particular,  for any $w\in \bW$, the set $\Omega_w$ is open in $G=\Omega_{e}$.
\item Let $P$ be a standard parabolic subgroup of $G$ and $w\in \bW$, then $PwP\cap \Omega_w$ is closed in $\Omega_w$. 
\end{enumerate}
\end{prop}

\begin{proof}
(1) We will show that $\Omega_w-\Omega_{w'}$ is closed in $\Omega_w$. Note that $ \Omega_w-\Omega_{w'}$ is a union of Bruhat cells. For $C(w_1)\subset \Omega_w-\Omega_{w'}$, we have $w_1>w$ but $w_1$ is not bigger than or equal to $w'$. We have $\ov{C(w_1)}=\coprod_{w_2\le w_1}C(w_2)$, and thus $$\ov{C(w_1)}\cap \Omega_w=\coprod_{w_1\ge w_2\ge w}C(w_2).$$
For $w_2$ with $w_1\ge w_2\ge w$, $w_2$ is not bigger than or equal to $w'$ (otherwise, we will have $w_1\ge w'$). Thus $C(w_2)\subset \Omega_w-\Omega_{w'}$. This implies that $\ov{C(w_1)}\cap \Omega_w \subset  \Omega_w-\Omega_{w'}$ and hence $\Omega_w-\Omega_{w'}$ is closed.

(2) Note that $PwP$ is a union of Bruhat cells. Thus it suffices to show that if $C(w_1)\subset PwP\cap \Omega_w$, then $\ov{C(w_1)}\cap \Omega_w\subset PwP\cap \Omega_w$. Given $C(w_2)\subset \ov{C(w_1)}\cap \Omega_w $, we need to show that $C(w_2)\subset PwP\cap \Omega_w$. Note that $C(w_2)\subset \ov{C(w_1)}\cap \Omega_w$ implies that $w_1\ge w_2\ge w$. By Proposition 2 of \cite{BKPST}, the set $D$ of elements $w'\in \bW$ such that $C(w')\subset PwP$ forms a Bruhat interval, i.e., there exists $w_{\min}, w_{\max}\in D$ such that $w'\in D$ if and only if $w_{\min}\le w'\le w_{\max}$. By the assumption $C(w_1)\subset PwP$, we get $w_1\in D$. Since $w\in D$ and $w_1\ge w_2\ge w$, we get $w_2\in D$, i.e., $C(w_2)\subset Pw_2P$. This completes the proof.
%Since $w_1\ge w_2$, we get $\ov{C(w_1)}\supset \ov{C(w_2)}$. Since $Pw_1P=PC(w_1)P$ is dense in $P\ov{C(w_1)}P$. We get $\ov{Pw_1P}=\ov{P\ov{C(w_1)}P}\supset \ov{P\ov{C(w_2)}P}=\ov{Pw_2P}$. Similarly, from $w_2\ge w$, we can get $\ov{Pw_2P}\supset \ov{PwP}$. Thus we get $\ov{Pw_1P}\supset \ov{Pw_2P}\supset \ov{PwP}$. On the other hand, we have $C(w_1)\subset PwP$, which implies $Pw_1P=PwP$. Thus $\ov{Pw_1P}= \ov{Pw_2P}=\ov{PwP} $
\end{proof}

For a character $\omega$ of $Z$, and a subspace $X$ of $G$ which is $Z$ invariant, we denote by $C_c^\infty(X,\omega)$ the space of smooth compactly supported functions $f$ on $X$ such that $f(zx)=\omega(z)f(x)$. Let $Y$ be a $Z$-invariant closed subset of $X$, we have the following exact sequence 
\begin{align}\label{eq2.3}
0\ra C_c^\infty(X-Y, \omega)\ra C_c^\infty(X,\omega)\ra C_c^\infty(Y,\omega)\ra 0,
\end{align}
where $C_c^\infty(X-Y,\omega)\ra C_c^\infty(X,\omega)$ is the map induced by zero extension, and the map $C_c^\infty(X,\omega)\ra C_c^\infty(Y,\omega) $ is the restriction map. See \cite[$\S$1.8]{BZ76}.

 Since $\Omega_w$ is open in $G$, we have $C_c^\infty(\Omega_w,\omega)\subset C_c^\infty(G,\omega)$. Since $C(w)$ is closed in $\Omega_w$, we have the exact sequence 
\begin{align}
\label{eq2.4}0\ra C_c^\infty(\Omega_w-C(w),\omega)\ra C_c^\infty(\Omega_w,\omega)\ra C_c^\infty(C(w),\omega)\ra 0.
\end{align}

\subsection{Weyl elements which support Bessel functions}

 Let $B(G)$ be the subset of $\bW$ which consists Weyl elements which can support Bessel functions, i.e., $w\in B(G)$ if and only if for every simple root $\alpha\in \Delta$ we have $w\alpha>0$ implies $w\alpha\in \Delta$. Recall that $w_{\ell}=\begin{pmatrix} &J_r\\ -J_r& \end{pmatrix}$, which represents the long Weyl element in $\bW$. Then we know that $w\in B(G)$ if and only if $w_{\ell}w$ is the long Weyl element of a standard Levi subgroup of $G$. For $w\in B(G)$, let $P_w=M_wN_w$ be the standard parabolic subgroup such that $w_{\ell}w=w_{\ell}^{M_w}$, where $M_w$ is the Levi subgroup of $P_w$ and $w_{\ell}^{M_w}$ is the long Weyl element in $M_w$. Let $\theta_w$ be the subset of $\Delta$ which consists all simple roots in $M_w$. Then we have 
$$\theta_w=\wpair{\alpha\in \Delta| w\alpha >0}\subset \Delta.$$
The assignment $w\mapsto \theta_w$ defines a bijection between the set $B(G)$ and subsets of $\Delta$. Given a subset $\theta\subset \Delta$, we will write $w_\theta$ for the corresponding element in $B(G)$. We then have $w_{\emptyset}=w_{\ell}, w_{\Delta}=e$.
\begin{lem}[Proposition 2.1 of \cite{CPSS05}]\label{lem2.6}
Let $w,w'\in B(G)$. Then $w' \le w$ iff $M_w\subset M_{w'}$ iff $\theta_w\subset \theta_{w'}$. 
\end{lem}
If $w,w'\in B(G)$ with $w>w'$, we set (following Jacquet \cite{Ja16}) 
$$d_B(w,w')=\max\wpair{m| \textrm{ there exist } w_i\in B(G) \textrm{ with } w=w_m>w_{m-1}>\cdots>w_0=w'}.$$
The number $d_B(w,w')$ is called the Bessel distance between $w$ and $w'$. We now consider all elements $w\in \bW$ with $d_B(w,e)=1$. By Lemma \ref{lem2.6}, if $d_B(w,e)=1$, then there exists a simple root $\gamma\in \Delta$ such that $w=w_{\Delta-\wpair{\gamma}}$. There are totally $r$ such Weyl elements. Actually one can check that they are represented by $\tilde w_n$: 
$$\tilde w_{n}=w_{\Delta-\wpair{\alpha_n}} .$$
Here we didn't distinguish an element in $G$ and the Weyl element it represents. 

For $w_1,w_2\in \bW$ with $ w_1<w_2$, denote by $[w_1,w_2]$ the closed Bruhat interval $\wpair{w\in \bW| w_1\le w\le w_2}$.

Recall that $Q_n=L_nV_n$ is the parabolic subgroup of $\Sp_{2r}$ with the Levi subgroup 
$$L_n=\wpair{\bm_r(a,a_{n+1},\dots, a_r), a\in \GL_n(F), a_i\in F^\times, n+1\le i\le r}.$$

\begin{lem}\label{lem2.7}
The set $w\in \bW$ such that $C(w)\subset Q_n \tilde w_n Q_n$ is the Bruhat interval $[w_{\min},w_{\max}]$ with $w_{\min}=\tilde w_n$ and $w_{\max}=w_{\ell}^{L_n}\tilde w_n $, where $w_\ell^{L_n}$ is the long Weyl element in $L_n \cong \GL_n(F)\times (\GL_1(F))^{r-n}$.
\end{lem}
\begin{proof}
By \cite[Proposition 2]{BKPST}, we know that the set $D=\wpair{w\in \bW| C(w)\subset Q_n \tilde w_n Q_n}$ is a Bruhat interval $[w_{\min},w_{\max}]$. Note that the set of simple roots in $L_n$ is $\tilde\theta=\wpair{\alpha_1,\dots,\alpha_{n-1}}$. Let $\bW_{\tilde \theta}$ be the Weyl group of $L_n$. It is known that $D=W_{\tilde \theta}\cdot \tilde w_n\cdot W_{\tilde \theta}$, see \cite[Corollaire 5.20]{BT65}. Since $\tilde\theta\subset \Delta-\wpair{\alpha_n}=\theta_{\tilde w_n}$, we have $\tilde w_n (\tilde \theta)>0 $. Since the Weyl element represented by $\tilde w_n$ has order 2, we have $\tilde w_n^{-1}(\tilde \theta)>0$. Then by \cite[Proposition 1.1.3]{Ca} or \cite[Proposition 2 (2)]{BKPST}, we have $w_{\min}=\tilde w_n$. On the other hand, by \cite[Corollary 3]{BKPST}, we know that every $w\in W_{\tilde \theta}\cdot \tilde w_n\cdot W_{\tilde \theta}$ can be uniquely written as $w=\tilde w_n w'$ with $w'\in W_{\tilde \theta}$, and $\ell(w)=\ell(\tilde w_n)+\ell(w')$, where $\ell$ denotes the length of a Weyl element. Thus $w_{\max}=\tilde w_n w_{\ell}^{L_n}$. One can check easily that $\tilde w_n w_{\ell}^{L_n}=w_{\ell}^{L_n} \tilde w_n$.
\end{proof}

For $w\in B(G)$, denote 
$$A_w=\wpair{a\in A| \alpha(a)=1 \textrm{ for all } \alpha\in \theta_w},$$
which is the center of $M_w$. Note that $M_e=G$ and $A_e=Z$.

%For $w\in \bW$, denote $U_w^+=\wpair{u\in U| wuw^{-1}\in U}, U_w^-=\wpair{u\in U| wuw^{-1}\in \ov{U}}$. Then we have $U=U_w^+\cdot U_w^-$ for any $w\in \bW$. Moreover, the map $(u,t,u')\mapsto utwu'$ from $U\times A\times U_w^-$ to $C(w)=BwB$ defines a homeomorphism. For $w\in B(G)$, let $P_w=M_wN_w$ be the associated parabolic subgroup. Then $U_w^-=N_w$ and $U_w^+=U_{M_w}$, where $U_{M_w}=U\cap M_w$ is the standard maximal unipotent subgroup of $M_w$.

\subsection{Cogdell-Shahidi-Tsai's theory on partial Bessel functions } In this subsection, we review certain basic properties of partial Bessel functions developed by Cogdell, Shahidi and Tsai \cite{CST17} recently, which has fundamental importance in the proof of our local converse theorem.
\begin{thm}[Cogdell-Shahidi-Tsai] \label{thm2.8}
\begin{enumerate}
\item Let $w\in B(G)$, $m>0$ and $f\in C_c^\infty(\Omega_w,\omega)$. Suppose $B_m(aw,f)=0$ for all $a\in A_w$. Then there exists $f_0\in C_c^\infty(\Omega_w-C(w), \omega)$ which only depends on $f$, such that for sufficiently large $m$ depending only on $f$, we have $B_m(g,f)=B_m(g,f_0)$ for all $g\in G$. 
\item Let $w\in B(G)$. Let $\Omega_{w,0}$ and $\Omega_{w,1}$ be $U\times U$ and $A$-invariant open sets of $\Omega_w$ such that $\Omega_{w,0}\subset \Omega_{w,1}$ and $\Omega_{w,1}-\Omega_{w,0}$ is a union of Bruhat cells $C(w')$ such that $w'$ does not support a Bessel function, i.e., $w'\notin B(G)$. Then for any $f_1\in C_c^\infty(\Omega_{w,1}, \omega)$ there exists $f_0\in C_c^\infty(\Omega_{w,0},\omega)$ such that for all sufficiently large $m$ depending only on $f_1$, we have 
$B_m(g,f_0)=B_m(g,f_1),$ for all $g\in G$.
\end{enumerate}
\end{thm}
\begin{proof}
Part (1) is \cite[Lemma 5.13]{CST17} and part (2) is \cite[Lemma 5.14]{CST17}. Note that although the results in \cite{CST17} were proved in a different setting, their proof is in fact general enough to include our case.
\end{proof}

\begin{cor}\label{cor2.9}
 Let $f_0,f\in C_c^\infty(G,\omega)$ with $W^f(1)=W^{f_0}(1)=1 $. Then for each $i$ with $1\le i\le r$, there exist functions $f_{\tilde w_i}\in C_c^\infty(\Omega_{\tilde w_i}, \omega)$ such that for sufficiently large $m$ (depending only on $f,f_0$) we have 
$$B_m(g,f)-B_m(g,f_0)=\sum_{i=1}^r B_m(g,f_{\tilde w_i}).$$
\end{cor}
This is the analogue of \cite[Proposition 5.3]{CST17}.
\begin{proof}
For $m\ge \max\wpair{C(f),C(f_0)}, t\in Z$, we have 
$$B_m(t,f)=\frac{1}{\vol(U_m)}\int_{U\times U_m}f(utu')\psi_U^{-1}(u)\psi_m^{-1}(u')dudu'=\omega(t)W^f(1)=\omega(t). $$
Similarly, we have $B_m(t,f_0)=\omega(t)$ for $t\in Z$.
Thus we get $B_m(t,f)=B_m(t,f_0)$ for all $t\in A_e=Z$ and $m\ge \max\wpair{C(f),C(f_0)}$. By applying Theorem \ref{thm2.8} (1) to the Weyl element $e$, there exists a function $f_1\in C_c^\infty(G-B, \omega)$ such that $B_m(g,f-f_0)=B_m(g,f_1)$ for all $g\in G$ and $m$ large. Denote 
$$\Omega_{e,0}=\bigcup_{w\in B(G), w\ne 1}\Omega_w=\bigcup _{w\in B(G), d_B(w,1)=1}\Omega_w=\bigcup_{1\le i\le r}\Omega_{\tilde w_i},$$
and $\Omega_{e,1}=G-B$. Then we have $\Omega_{e,0}\subset \Omega_{e,1}$, and $\Omega_{e,1}-\Omega_{e,0}$ is a union of Bruhat cells $C(w')$ such that $w'\notin B(G)$.

By Theorem \ref{thm2.8} (2), there exists a function $f_2\in C_c^\infty(\Omega_{e,0},\omega)$ such that 
$$B_m(g,f-f_0)=B_m(g,f_1)=B_m(g,f_2), $$
for all $g\in G$ and $m$ large enough. As in the proof of \cite[Proposition 5.3]{CST17}, we can follow Jacquet \cite{Ja16} to write
$$f_2=\sum_{i}f_{\tilde w_i} \textrm{ with } f_{\tilde w_i}\in C_c^\infty(\Omega_{\tilde w_i}, \omega)$$
using a partition of unity argument. We then get 
$$B_m(g,f)-B_m(g,f_0)=\sum_{i=1}^r B_m(g,f_{\tilde w_i}), $$
for all $g\in G$ and $m$ large (which only depends on $f$ and $f_0$).
\end{proof}

\subsection{Outline of the proof of the local converse theorem}
We now repeat the main theorem of this paper. 
\begin{thm}\label{thm2.10}
Let $\pi,\pi_0$ be two irreducible $\psi_U$-generic supercuspidal representations of $\Sp_{2r}(F)$ with the same central character. If $\gamma(s,\pi\times \tau,\psi)=\gamma(s,\pi_0\times \tau,\psi) $ for all irreducible generic representations $\tau$ of $\GL_k(F)$ and for all $k$ with $1\le k\le r$, then $\pi\cong \pi_0$.
\end{thm}
When $r=1$, the above theorem is proved in \cite{ChZ, Zh3}.\\

We outline our proof of the above theorem. 

Let $\pi$ be an irreducible $\psi_U$-generic supercuspidal representation of $G=\Sp_{2r}(F)$ with central character $\omega$. We have $\CM(\pi)\subset C_c^\infty(G,\omega)$. We can consider the linear functional $\CM(\pi)\ra C^\infty(G,\psi_U,\omega)$, $f\mapsto W^f$ defined by
$$W^f(g)=\int_U \psi_U^{-1}(u)f(ug)du.$$
Since $\pi$ is assumed to be $\psi_U$-generic, this linear functional $f\mapsto W^f$ is non-zero. Thus we can take $f\in \CM(\pi)$ such that $W^f(1)=1$ and then consider the partial Bessel functions $B_m(g,f)$ for $m>0$. 

Let $\pi,\pi_0$ be two irreducible $\psi_U$-generic supercuspidal representations with the same central character $\omega$.  Given a positive integer $k$, denote by $\CC(k)$ the following condition:
  \begin{align*} &\gamma(s,\pi\times \tau, \psi)=\gamma(s,\pi_0\times \tau,\psi),\\
  & \textrm{ for all irreducible generic representations } \tau \textrm{ of } \GL_i(F),\\
  & \textrm{ and for all } i, \textrm{ with } 1\le i\le k.
  \end{align*}
  
  To make some statements cleaner, we will denote by $\CC(0)$ the condition that $\pi$ and $\pi_0$ have the same central character. The condition $\CC(0)$ is always assumed. 

We fix $f\in \CM(\pi),f_0\in \CM(\pi_0)$ such that $W^f(1)=W^{f_0}(1)=1$. We are going to show that the condition $\CC(r)$ implies $B_m(g,f)=B_m(g,f_0)$ for all $g\in G$ and $m$ large enough (which only depends on $f,f_0$). This will imply that $\pi\cong \pi_0$ by irreducibility and the uniqueness of Whittaker model. 

Corollary \ref{cor2.9} says that the condition $\CC(0)$ implies there exist functions $f_{\tilde w_i}\in \Omega_{\tilde w_i}$ for $1\le i\le r$ such that 
 \begin{align}\label{eq2.5}B_m(g,f)-B_m(g,f_0)=\sum_{i=1}^r B_m(g,f_{\tilde w_i}),\end{align}
for all $g\in G$ and large enough $m$ depending on $f,f_0$ only. In $\S$4, we will use induction to show that for $1\le k\le r-1$ the condition $\CC(k)$ implies that there exist functions $f_{\tilde w_i}'\in C_c^\infty(\Omega_{\tilde w_i},\omega)$ with $ k+1\le i\le r$ such that 
$$B_m(g,f)-B_m(g,f_0)=\sum_{i=k+1}^r B_m(g,f'_{\tilde w_i}). $$ 
Thus the condition $\CC(r-1)$ implies that $$B_m(g,f)-B_m(g,f_0)=B_m(g,f_{\tilde w_r}),$$
for some $f_{\tilde w_r}\in C_c^\infty(\Omega_{\tilde w_r},\omega)$. In $\S$5, we will show that the condition $\CC(r)$ implies that $B_m(g,f)-B_m(g,f_0)=0$ for all $g\in G$ and $m$ large. We have to deal with the cases $1\le k\le r-1$ and $k=r$ separately because the local zeta integrals in these two cases are different, see Eq.(\ref{eq1.1}).

\section{Preparations for the proof of the main theorem}\label{sec5}
\subsection{Sections of induced representations}
Let $\tau$ be an irreducible generic representation of $\GL_n(F)$ and $\psi$ be an unramified character of $F$. In this section, we construct certain sections in the induced representation $\tilde I(s,\tau,\psi^{-1})$ which will be used in the proof of our main theorem. These sections appeared in \cite{ChZ} when $n=1$ and in \cite{Zh2} when $n=2$. The ideas of these constructions go back to \cite{Ba95}.

Recall that $P_n=M_nN_n$ denotes the Siegel parabolic subgroup of $\Sp_{2n}(F)$. Let $\ov{P}_n=M_n\ov{N}_n$ be the opposite of $P_n$. For $b\in \Mat_{n\times n}(F)$ with $bJ_n=J_n{}^t\! b$, denote 
$$\bar \bn_n(b)=\begin{pmatrix}I_n & \\ b &I_n \end{pmatrix}.$$
Then $\ov{N}_n=\wpair{\bar \bn_n(b)| b\in \Mat_{n\times n}(F), b J_n=J_n {}^t\! b}.$
For a positive integer $i$
$$ \ov{N}_{n,i}=\wpair{\bar \bn_n(b)| \bar \bn_r\begin{pmatrix}0&0\\ b&0 \end{pmatrix}\in H^r_i }.$$ 
Let $D$ be an open compact subgroup of $N_n$ and $i$ be a positive integer. For $x\in D$, we consider the set 
$$S(x,i)=\wpair{ \bar y\in \ov{N}_n| \bar yx\in P_n\cdot \ov{N}_{n,i}}. $$
For a positive integer $c$, denote $K_c^{\GL_n}=I_n+\Mat_{n\times n}(\fp^c)$. 
\begin{lem}\label{lem3.1}
The following statements hold. 
\begin{enumerate}
\item For any positive integer $c$, there exists an integer $i_1=i_1(D,c)$ such that for all $i\ge i_1, x\in D, \bar y\in S(x,i)$, we can write 
$$\bar yx=\bn_n(b)\bm_n(a) \bar y_0, $$
with $\bn_n(b)\in N_n, a\in K_c^{\GL_n}$ and $\bar y_0\in \ov{N}_{n,i}$.
\item There exists an integer $i_2=i_2(D)$ such that $ S(x,i)=\ov{N}_{n,i}$ for all $x\in D$ and $i\ge i_2$.
\end{enumerate}
\end{lem}
See \cite[Lemma 4.1]{Ba95} for a similar statement for the group $\GL_n$.
\begin{proof}

For $x\in D, \bar y\in S(x,i)$, we assume that $\bar yx=p\bar \bn_n(b_0)^{-1}$ for $p\in P_n, \bar \bn_n(b_0)\in \ov{N}_{n,i}$. We have $\bar y^{-1}p=x\bar \bn_n(b_0)$. By abuse of notation, we write $\bar y^{-1}=\bar \bn_n(y),y\in \Mat_{n\times n}(F)$. Let $p=\bn_n(b')\bm_n(a)$. We have 
$$ \bar y^{-1}p=\begin{pmatrix}a& b'a^*\\ ya & (yb'+I_n)a^* \end{pmatrix}.$$
On the other hand, we denote $x=\bn_n(b)$ with $b\in \Mat_{n\times n}(F)$. We have 
$$x \bar \bn_n(b_0)=\begin{pmatrix}I_n+bb_0 & b\\ b_0 &I_n \end{pmatrix}.$$
From the equality $ \bar y^{-1}p=x\bar \bn_n(b_0)$, we get 
$$a=I_n+bb_0, \textrm{ and }  y=b_0a^{-1}=b_0(I_n+bb_0)^{-1}.$$
Since the entries of $b$ are bounded and the entries of $b_0$ go to zero as $i\ra \infty$, it follows that for any positive integer $c$, we can take $i_1=i_1(D,c)$ such that if $i\ge i_1$, we have $a\in K_c^{\GL_n}$. This proves (1). 

Next we show that $\bar y^{-1}=\bar \bn_n(y)\in \ov{N}_{n,i}$. We have $ y=b_0a^{-1}=b_0(I_n+bb_0)^{-1}$. Since $\bar \bn_n(b_0)\in \ov{N}_{n,i}$, we have 
$$b_0\in \begin{pmatrix}\fp^{(4r-2n+1)i} & \fp^{(4r-2n-1)i}& \dots &\fp^{(4r-4n+3)i} \\ \dots &\dots & \dots &\dots \\ \fp^{(4r-1)i}& \fp^{(4r-3)i} &\dots &\fp^{(4r-2n+1)i} \end{pmatrix}.$$
 Since $b$ has bounded entries, we assume that $b\in \Mat_{n\times n}(\fp^{-d})$ for a positive integer $d$. Then we have 
 $$bb_0\in \begin{pmatrix}\fp^{(4r-4n+3)i-d} & \fp^{(4r-4n+3)i-d} &\dots & \fp^{(4r-4n+3)i-d}\\ \dots &\dots & \dots & \dots  \\ \fp^{(4r-2n+1)i-d}& \fp^{(4r-2n+1)i-d} &\dots &\fp^{(4r-2n+1)i-d}\end{pmatrix}. $$
 Thus we have 
 $$ b_0bb_0\in \begin{pmatrix}\fp^{(8r-6n+4)i-d} & \fp^{(8r-6n+4)i-d} &\dots & \fp^{(8r-6n+4)i-d}\\ \dots &\dots & \dots & \dots  \\ \fp^{(8r-4n+2)i-d}& \fp^{(8r-4n+2)i-d} &\dots &\fp^{(8r-4n+2)i-d}\end{pmatrix}. $$
We can take an integer $i'_2(D)$ such that for $i\ge i'_2(D)$
$$ b_0bb_0 \in \begin{pmatrix}\fp^{(4r-2n+1)i} & \fp^{(4r-2n-1)i}& \dots &\fp^{(4r-4n+3)i} \\ \dots &\dots & \dots &\dots \\ \fp^{(4r-1)i}& \fp^{(4r-3)i} &\dots &\fp^{(4r-2n+1)i} \end{pmatrix},$$
i.e., $\bar \bn_n(b_0bb_0)\in \ov{N}_{n,i}$. One can check that the same argument implies that for $i\ge i_2(D)$, we have $\bn_n(b_0(bb_0)^k)\in \ov{N}_{n,i}$. Since 
$$ y=b_0(I_n+bb_0)^{-1}=b_0-b_0(bb_0)+b_0(bb_0)^2+\cdots $$
 we get $\bar \bn_n(y)\in \ov{N}_{n,i}$. This shows that $ S(x,i)\subset \ov{N}_{n,i}$. 
 
 We next show that there exists an integer $i_2''(D)$ such that $\ov{N}_{n,i}\subset S(x,i)$ for $i\ge i_2''(D)$. We write $x=\bn_n(b)$ as above. For $\bar \bn(b_0)\in \ov{N}_{n,i}$, we want to show that $\bar \bn(b_0)\in S(x,i)$. For $i$ large, we can assume the $\det(I_n+b_0b)\ne 0$. From this assumption, we can check that $\bar \bn(b_0) x\in P\cdot \ov{N}$. Thus we can write 
 $$\bar \bn(b_0) x=p  \bar \bn_n(y). $$
 We need to show that $\bar \bn_n(y)\in \ov{N}_{n,i}$ for $i$ large. The argument is the same as the proof of $S(x,i)\subset \ov{N}_{n,i}$. We omit the details.
\end{proof}

Note that the double cover $\wt{\Sp}_{2n}$ splits over $N_n$ but not over $\ov{N}_n$ in general.  Let $K_0^n=\Sp_{2n}(\fo)$ and $K_m^n=(1+\Mat_{(2n)\times (2n)}(\fp^m))\cap K$. Since in this subsection we only consider subgroups of $\Sp_{2n}$, we will drop $n$ from the notation $K_m^n$ and write it as $K_m$. It is known that there exists an integer $m_0\ge 0$ such that the double cover $\wt{\Sp}_{2n}(F)\ra \Sp_{2n}(F)$ splits over $K_{m_0}$, see \cite[Lemma 3, p.959]{K}, i.e., there exists a continuous open map $s: K_{m_0}\ra \wt{\Sp}_{2n}(F)$ such that it is a group homomorphism and the composition $K_{m_0}\ra \wt{\Sp}_{2n}(F)\ra \Sp_{2n}(F)$ is the inclusion map. When the residue characteristic is odd, we can take $m_0=0$, see \cite[p.43]{MVW}, and the splitting is known to be unique, see \cite[p.1662]{GS} for example. In general, there exists an integer $m_1>m_0$ such that $s|_{K_{m_1}}$ is unique. In fact, a splitting $s:K_{m_0}\ra \wt{\Sp}_{2n}(F)$ has the form $s(k)=(k,\epsilon(k))$ for a cocylce $\epsilon: K_{m_0}\ra \wpair{\pm 1}$. From this one can check that any two splittings differ by a quadratic character of $K_{m_0}$. Thus it suffices to show that any quadratic character of $K_{m_0}$ vanishes on $K_{m_1}$ for certain $m_1\ge m_0$. To do so, we consider the square map $K_{m_0}\ra K_{m_0}, x\mapsto x^2$. The image of the square map is open and thus contains a $K_{m_1}$ for some $m_1\ge m_0$. It's clear that any quadratic character of $K_{m_0}$ is trivial on $K_{m_1}$.

Let $(\tau,V_\tau)$ be an irreducible generic representation of $\GL_n(F)$. For $i>0$ and $v\in V_\tau$, we consider the function $f_s^{i,v}:\wt{\Sp}_{2n}(F)\ra V_\tau$ by 
$$ f_s^{i,v}(\tilde g)=\left\{\begin{array}{lll}\epsilon\gamma^{-1}_{\psi^{-1}}(\det(a))\delta_{P_n}^{1/2}(a)|\det(a)|^{s-1/2}\tau(a)v, &\textrm{ if } \tilde g=(\bn_n(b)\bm_n(a) , \epsilon)s(\bar y) , \bn_n(b)\in N_n, \\
& a\in \GL_n(F), \bar y\in \ov{N}_{n,i},\epsilon\in \wpair{\pm 1}.\\
0, &\textrm{ otherwise.}  \end{array}\right.$$

\begin{lem}\label{lem3.2}
For any fixed $v\in V_\tau$, there exists an integer $i_0(v)$ such that if $i\ge i_0(v)$, $f_s^{i,v}$ defines an element in $\tilde I(s,\tau,\psi)$.
\end{lem}
\begin{proof}
In the group $\wt{\Sp}_{2n}(F)$ we have
\begin{align*}&(\bn_n(b_0)\bm_n(a_0),\epsilon_0)(\bn_n(b)\bm_n(a),\epsilon)\\
=&(\bn_n(b_0) \bm_n(a_0) \bn_n(b)\bm_n(a_0)^{-1},1) (\bm_n(a_0a),\epsilon_0 \epsilon (\det(a_0),\det(a))_F).
\end{align*}
On the other hand, we have $\gamma_{\psi^{-1}}^{-1}(\det(a_0a))=\gamma_{\psi^{-1}}^{-1}(\det(a_0))\gamma_{\psi^{-1}}^{-1}(\det(a))(\det(a_0),\det(a))_F$. From these formulas, we can check that
 $$f_s^{i,v}((\bn_n(b_0) \bm_n(a_0), \epsilon_0)\tilde g)=\epsilon_0 \gamma_{\psi^{-1}}^{-1}(\det(a_0)) \delta_{P_n}^{1/2}(a_0)|\det(a_0)|^{s-1/2}f_s^{i,v}(\tilde g). $$

 Next we need to check that $f_{s}^{i,v}$ is right invariant under an open compact subgroup of $\wt{\Sp}_{2n}(F)$. For $v\in V_\tau$, we can find a positive integer $c$ such that $v$ is fixed by $1+M_{n\times n}(\fp^c)\subset \GL_n(F)$ by smoothness of $\tau$. We require $c$ large such that $1+\fp^c\subset F^{\times ,2}$ (such $c$ exists since $F^{\times,2}$ is an open subgroup of $F^\times$). Let $i_0(v)=\max\wpair{c, i_1(K_c\cap N_n, c ), i_2(K_c\cap N_n)}$ and $\ov{N}_{n,i_0}\subset K_{m_1}$. Now take an integer $i$ such that $i\ge i_0$.  Let $m\ge m_1$ be large such that $\ov{N}_n\cap K_m\subset \ov{N}_{n,i}$. Note that $m\ge m_1$ implies that the splitting $s:K_{m}\ra \wt{\Sp}_{2n}(F)$ exists and is unique. On the other hand, since $m_1$ is fixed, $m$ only depends on $i$.
 
 We will show that $f_{s}^{i,v}(\tilde g s(k))=f_s^{i,v}(\tilde g)$ for all $k\in K_m$, for $i\ge i_0(v)$. We have the decomposition 
 $$K_m= (N_n\cap K_m )(M_n\cap K_m)(\ov{N}_n \cap K_m).$$
 For $k\in \ov{N}_n\cap K_m\subset \ov{N}_{n,i}$, we have $ f_{s}^{i,v}(\tilde g s(k))=f_s^{i,v}(\tilde g)$ by the definition of $f_s^{i,v}$. 
 
 For $\bm_n(a)\in M_n\cap K_m$, we have $\det(a)\in 1+\fp^m\subset F^{\times,2}$. Thus $c(\bm_n(a),\bm_n(a_0))=(\det(a),\det(a_0))_F=1$ for $\bm_n(a),\bm_n(a_0)\in M_n\cap K_m$. It follows that $a\mapsto (a,1)$ from $M_n\cap K_m$ to $\wt{\Sp}_{2n}(F)$ is a group homomorphism. By the uniqueness of the splitting, we get $s(\bm_n(a))=(\bm_n(a),1)$. 
 
 For $\bm_n(a_0)\in M_n\cap K_m$, and $\bn_n(b)\bm_n(a)\in P$, we have $c(\bn_n(b)\bm_n(a),\bm_n(a_0))=(\det(a),\det(a_0))_F=1$, and thus
 \begin{align*}
 (\bn_n(b)\bm_n(a),\epsilon)s(\bar y)s(\bm_n(a_0))&=(\bn_n(b)\bm_n(a) \bm_n(a_0),\epsilon)s( \bm_n(a_0)^{-1}\bar y \bm_n(a_0)).
 \end{align*}
 We can check that $\bm_n(a_0)^{-1}\bar y \bm_n(a_0)\in \ov{N}_{n,i} $. Thus 
 \begin{align*}
& f_{s}^{i,v}((\bn_n(b)\bm_n(a),\epsilon)s(\bar y)s(\bm_n(a_0)))\\
 =&\epsilon \gamma_{\psi^{-1}}^{-1}(\det(aa_0))\delta_P^{1/2}(aa_0)|\det(aa_0)|^{s-1/2}\tau(aa_0)v\\
 =&\epsilon \gamma_{\psi^{-1}}^{-1}(\det(a))\delta_P^{1/2}(a)|\det(a)|^{s-1/2}\tau(a)v\\
 =&f_{s}^{i,v}((\bn_n(b)\bm_n(a),\epsilon)s(\bar y)),
 \end{align*}
 where we used $|\det(a_0)|=1$ and $\tau(a_0)v=v$. 
 
 Next, we consider $k\in N_n\cap K_m\subset N\cap K_c$. By assumption and the above lemma, we get 
 $$S(k,i)=S(k^{-1},i)=\ov{N}_{n,i}.$$
 In particular, for $\bar y\in \ov{N}_{n,i}$, we have $\bar y k\in P_n\cdot \ov{N}_{n,i}$ and $\bar y k^{-1}\in P_n\cdot \ov{N}_{n,i}$. Thus $g\in P_n\cdot \ov{N}_{n,i}$ if and only if $gk\in  P_n\cdot \ov{N}_{n,i}$. Moreover, by (1) of last lemma, we can write $\bar y k=\bn_n(b_0)\bm_n(a_0) \bar y'$ with $a_0\in 1+\Mat_{n\times n}(\fp^c)$, which fixes $v$. From this, we can check that $f_s^{i,v}(\tilde g s(k))=f_s^{i,v}(\tilde g)$. 
\end{proof}

Let $i\ge i_0(v)$, we can consider the $\BC$-valued function $\xi_s^{i,v}$ on $\wt{\Sp}_{2n}\times \GL_n$ associated with $f_s^{i,v}$
$$\xi_s^{i,v}( \tilde g, a)=\lambda(\tau(a)f_s^{i,v}(\tilde g)),$$
for a fixed $\psi^{-1}$-Whittaker functional $\lambda$ of $\tau$.
We have 
\begin{equation}\label{eq3.1}
\xi_s^{i,v}(\tilde g,I_n)=\left\{\begin{array}{lll}\epsilon \gamma_{\psi^{-1}}^{-1}(\det(a))\delta_{P_n}^{1/2}(a)|\det(a)|^{s-1/2}W_v(a)&\textrm{ if } \tilde g=(\bn_n(b)\bm_n(a) , \epsilon)s(\bar y) , \\&\bn_n(b)\in N_n,  a\in \GL_n(F),\\ & \bar y\in \ov{N}_{n,i},  \epsilon\in \wpair{\pm 1}.\\
0, &\textrm{ otherwise.}  \end{array}  \right.
\end{equation}

Denote $\tilde \xi_{1-s}^{i,v}=M(s,\tau,\psi^{-1})\xi_s^{i,v}\in \tilde V(1-s,\tau^*,\psi^{-1})$. Let $D$ be an open compact subset of $N_n$, we evaluate $\tilde \xi_{s}^{i,v}(w_n x)$ for $x\in D$.
\begin{lem}\label{lem3.3}
There is an integer $I(D,v)\ge i_0(v)$ such that for all $i\ge I(D,v)$, we have $\tilde \xi_{1-s}^{i,v}(w_nx,I_n)=\vol(\ov{N}_{n,i})v$ for all $x\in D$. 
\end{lem}
\begin{proof}
Let $d_n=\diag(-1,1,-1,\dots)\in \GL_n$. Recall that 
 \begin{align*}\tilde \xi_{1-s}^{i,v}(w_nx,I_n)&=\int_{N_n}\xi_s^{i,v}(w_n^{-1}uw_n x, d_n )du \\
                                                                     &=\int_{N_n}\xi_s^{i,v}(d_n w_n^{-1}uw_n x, I_n )du.\end{align*}
Let $c$ be a positive integer such that $v$ is fixed by $1+\Mat_{n\times n}(\fp^c)$ under the action of $\tau$, and $1+\fp^c\subset F^{\times ,2}$.   Let $I(D,v)=\max\{ i_0(v), i_1(D,c), i_2(D)\}$.   By definition of $\xi_s^{i,v}$ and Lemma \ref{lem3.1}, 
$ \xi_s^{i,v}(d_n w_n^{-1}uw_n x, I_n )\ne 0$ if and only if $d_n w_n^{-1} u w_n x \in P_n \cdot \ov{N}_{n,i}$ if and only if $d_n w_n^{-1}uw_n \in \ov{N}_{n,i}$.    On the other hand, for $d_nw_n^{-1}uw_n \in \ov{N}_{n,i}=S(x,i) $, by Lemma \ref{lem3.1}(1), we can write 
$$ d_nw_n^{-1}uw_n=\bn_n(b)\bm_n(a)\bar y, \textrm{ with } a\in 1+\Mat_{n\times n}(\fp^c), \bar y\in \ov{N}_{n,i},$$
and thus $$\xi_s^{i,v}(d_n w_n^{-1}uw_n x, I_n )=\tau(a)v=v,$$
where we used $\gamma_{\psi^{-1}}^{-1}(\det(a))=1$ (since $\det(a)\in 1+\fp^c\subset F^{\times,2}$), and $|\det(a)|=1$. We now get 
$$\xi_s^{i,v}(d_nw_n^{-1}uw_n, I_n)=\left\{ \begin{array}{lll}v, & \textrm{ if } d_nw_n^{-1}uw_n\in \ov{N}_{n,i}\\0, & \textrm{ otherwise}. \end{array}\right.$$  
Now the assertion follows.                                                   
\end{proof}
Since $ \tilde \xi_{1-s}^{i,v} \in\tilde V(1-s,\tau^*,\psi^{-1})$, we get 
\begin{equation}\label{eq3.2}
\tilde \xi_{1-s}^{i,v}((\bn_n(b) \bm_n(a),\epsilon) (w_nx, 1) )=\vol(\ov{N}_{n,i})\epsilon  \delta_{P_n}^{1/2}(a)|\det(a)|^{1/2-s}\gamma_{\psi^{-1}}^{-1}(\det(a)) W^*_v(a),
\end{equation}
for $x\in D, i\ge I(D,v)$, where $W_v^*$ is the Whittaker function of the representation $\tau^*$ associated with $v$. 

\subsection{The Schwartz functions $\phi_{m,r}^n$, $n<r$} 
As usual we fix the positive integer $r$ and the group $G=\Sp_{2r}(F)$. For positive integers $m,n$ with $1\le n<r$, we consider the function $\phi_{m,r}^n\in \CS(F^n)$ defined by
$$\phi_{m,r}^n(x_1,\dots,x_n)=\Char_{\fp^{(2r-1)m}}(x_1)\Char_{\fp^{(2r-3)m}}(x_2)\cdots \Char_{\fp^{(2r-2n+1)m}}(x_n).$$
Recall that $U_m=H_m\cap U^r$. Denote $$N_{n,m}'=\wpair{\bn_n(b)| \bn_{r} \begin{pmatrix} 0&b\\ 0_{(r-n)\times (r-n)}&0 \end{pmatrix} \in U_m}. $$
As usual, let $\psi$ be an unramified additive character of $F$. We then have the Weil representation $\omega_{\psi^{-1}}$ of $\wt{\Sp}_{2n}$ on $\CS(F^n)$. In the following, we will write $\phi_{m,r}^n$ as $\phi_m^n$ since $r$ is fixed.

\begin{lem}\label{lem3.4} We have $\omega_{\psi^{-1}}(\bn_n(b))\phi^n_m=\phi^n_m$ for all $\bn_n(b)\in N_{n,m}'.$
\end{lem}
\begin{proof}

We have $$\omega_{\psi^{-1}}(\bn_n(b))\phi_m(x)=\psi(xJ_n{}^t\! b {}^t\! x )\phi_m(x).$$
For $x\in \Supp(\phi_m), \bn_n(b)\in U_m$, we can check that $xJ_n{}^t\! b {}^t\! x\in \fp^m $ and thus the assertion follows. We omit the details. 
\end{proof}

\subsection{The Schwartz function $\phi_{m,r}^n$, $n=r$}

We consider the function $\phi_{m,r}^r\in \CS(F^r)$ defined by
$$\phi_{m,r}^r(x_1,\dots, x_r)=\Char_{\fp^{(2r-1)m}}(x_1) \Char_{\fp^{(2r-3)m}}(x_2)\cdots \Char_{\fp^{3m}}(x_{r-1}) \Char_{1+\fp^m}(x_r).$$
We will omit $r$ from the notation if $r$ is understood. Thus we will sometimes write $\phi_{m,r}^r$ as $\phi_m$.

Let $\psi$ be an unramified additive character of $F$, we then consider the Weil representation $\omega_{\psi^{-1}}$ of $\wt{\Sp}_{2r}$. 
\begin{lem}\label{lem3.5}
\begin{enumerate}
\item For $u\in N_r\cap H_m$, we have $\omega_{\psi^{-1}}(u)\phi_m=\psi_U^{-1}(u)\phi_m$.
\item Given $a=(a_{ij})\in \GL_r(F)$. If $a_{ri}\in \fp^{-(2i-1)m}$ for all $i$ with $1\le i\le r$, then $\omega(w_r)\phi_m(e_r a)\ne 0$.
\end{enumerate}
\end{lem}
Recall that $e_r=(0,0,\dots,0,1)\in F^r$. 
\begin{proof}
(1) Write $u=\bn_r(b)\in N_r\cap H_m$. Then we have 
$$b=(b_{ij})\in \begin{pmatrix} \fp^{-(2r-1)m} & \fp^{-(2r+1)m} & \dots & \fp^{-(4r-3)m}\\
 \dots & \dots & \dots & \dots \\
 \fp^{-3m} & \fp^{-5m} &\dots & \fp^{-(2r+1)m}\\
 \fp^{-m} & \fp^{-3m} &\dots &\fp^{-(2r-1)m} \end{pmatrix}.$$
We have 
$$\omega_{\psi^{-1}}(u)\phi_m(x)=\psi^{-1}(xbJ_r {}^t\! x)\phi_m(x).$$
For $x=(x_1,\dots,x_n)\in \Supp(\phi_m)$, we can check that $xbJ_r {}^t\! x \equiv b_{n1}x_n^2\equiv b_{n1}\mod \fo$. Thus 
$$\omega_{\psi^{-1}}(u)\phi_m=\psi^{-1}(b_{n1})\phi_m=\psi^{-1}_U(u)\phi_m. $$
(2) Denote $x=e_ra=(a_{r1},\dots ,a_{rr})\in F^r$. We have 
\begin{align*}
&\quad \omega_{\psi^{-1}}(w_n)\phi_m(x)\\
&=\beta_{\psi^{-1}} \int_{F^r}\phi_m(y)\psi^{-1}(2yJ_r{}^t\! x)dy\\
&=\beta_{\psi^{-1}} \int_{\fp^{(2r-1)m}}\psi^{-1}(2y_1a_{rr})dy_1 \int_{\fp^{(2r-3)m}}\psi^{-1}(2y_2a_{r(r-1)})dy_2 \dots \int_{1+\fp^m}\psi^{-1}(2y_r a_{r1})dy_r.
\end{align*}
Now the assertion follows from the fact that $\psi^{-1}$ has conductor $\fo$.
\end{proof}

\subsection{A result of Jacquet-Shalika}

\begin{prop}\label{prop3.6}
Let $W'$ be a smooth function on $\GL_n(F)$ which satisfies $W'(ug)=\psi(u)W'(g)$ for all $u\in U_{\GL_n}$ and for each $m$, the set $\wpair{g\in \GL_n| W'(g)\ne 0, |\det(g)|=q^m} $ is compact modulo $U_{\GL_n}$. Assume, for all irreducible generic representation $\tau$ of $\GL_n(F)$ and for all Whittaker functions $W\in \CW(\tau,\psi^{-1})$, the following integral vanishes for $\Re(s)<<0$, 
$$\int_{U_{\GL_n}\setminus \GL_n}W'(g)W(g)|\det(g)|^{-s-k}=0,$$
where $k$ is a fixed number, then $W'\equiv 0$.
\end{prop}

This is a corollary of \cite[Lemma 3.2]{JaS85}. One can find an argument that \cite[Lemma3.2]{JaS85} implies the current form of the above proposition in \cite[Corollary 2.1]{Chen}.

\section{Inductive Step}\label{sec6}
We fix the notations as in $\S$2.5. To avoid ambiguity, we repeat part of the notations. We fixed two $\psi_U$-generic irreducible representations $\pi,\pi_0$ of $\Sp_{2r}(F)$ with the same central character $\omega$. Let $f\in \CM(\pi),f_0\in \CM(\pi_0)$ such that $W^f(1)=W^{f_0}(1)=1$. We have defined partial Bessel functions $B_m(g,f)$ and $B_m(g,f_0)$.

In this section, we will use induction to prove the following 
\begin{prop}\label{prop4.1}
Given an integer $n$ with $0\le n <r$. The condition $\CC(n)$ implies that there exist functions $f_{\tilde w_i}\in C_c^\infty(\Omega_{\tilde w_i},\omega), n+1\le i\le r$ such that 
$$B_m(g,f)-B_m(g,f_0)=\sum_{i=n+1}^rB_m(g,f_{\tilde w_i}),$$
for all $g\in G$ and all sufficiently large $m$ depending only on $f,f_0$.
\end{prop}
The base case when $n=0$ is just Corollary \ref{cor2.9}. We now consider the induction step. Let $n$ be an integer such that $1\le n<r$.  We assume that we know $\CC(n-1)$ implies that there exist $f_{\tilde w_i}\in C_c^\infty(\Omega_{\tilde w_i},\omega),$ such that 
\begin{equation}\label{eq4.1}B_m(g,f)-B_m(g,f_0)=\sum_{i=n}^rB_m(g,f_{\tilde w_i}),\end{equation}
for all $g\in G$ and large enough $m$.

Recall that $Q_n=L_nV_n$ denotes the standard parabolic subgroup of $\Sp_{2r}(F)$ with Levi subgroup $L_n=\wpair{\bm_r(a,a_{n+1},\dots,a_n),a\in \GL_n(F), a_{i}\in F^\times, n+1\le i\le n}$.

Denote $$E^-_n=\wpair{\begin{pmatrix}I_n & x&y \\ &I_{2r-2n}&x' \\ &&I_n \end{pmatrix}\in \Sp_{2r}},$$
and $$E^+_n=\wpair{\begin{pmatrix}u_1 &&\\ &u_2& \\ &&u_1' \end{pmatrix}\in \Sp_{2r}, u_1\in U_{\GL_n}, u_2\in U_{\Sp_{2r-2n}}}. $$
Then $U=E^+_n\cdot E^-_n, E_n^-\subset V_n$. Moreover, by \cite[p.12]{Ca}, the product map 
$$Q_n \times \wpair{\tilde w_n}\times E_n^- \ra Q_n \tilde w_n Q_n$$
induces an isomorphism. 

Recall that $U_m$ denotes $U\cap H_m$.
\begin{lem}
For $u_0\in E_n^--(E_n^-\cap U_m), u^+\in E_n^+\cap U_m$, then $u_0':=(u^+)^{-1}u_0(u^+)\in E_n^--(E_n^-\cap U_m) .$
\end{lem}
\begin{proof}
 We suppose that 
$$u_0=\begin{pmatrix}I_n &x &y \\ &I_{2r-2n}&x' \\ &&I_n \end{pmatrix}, u^+=\begin{pmatrix}u_1&&\\ &u_2&\\ &&u_1'  \end{pmatrix},$$
with $x\in \Mat_{n\times (2r-2n)}, y\in \Mat_{n\times n}, u_1\in U_{\GL_n}, u_2\in U_{\Sp_{2r-2n}}$. 
Then 
$$ u_0'=\begin{pmatrix}I_n & u_1^{-1}xu_2 & u_1^{-1}y u_1' \\ &I_{2r-2n} & u_2^{-1}x'u_1' \\ &&I_n \end{pmatrix}.$$
From a detailed analysis of the exponents in each entry of the matrix, one can show that $u_0'\notin E_n^-\cap U_m$. We omit the details. 
\end{proof}

\begin{lem}\label{lem4.3}
\begin{enumerate}
\item We have 
$$B_m(\bm_r(a),f)-B_m(\bm_r(a),f_0)=0, a\in \GL_r(F) ,$$
for large enough $m$ depending only  on $f,f_0$.
\item We have $Q_n \tilde w_n Q_n\cap \Omega_{\tilde w_i}=\emptyset$ for $i\ge n+1$. In particular, we have $$B_m(g,f_{\tilde w_i})=0,$$
for all $g\in Q_n\tilde w_n Q_n$ and $i\ge n+1$, and thus 
$$B_m(g,f)-B_m(g,f_0)=B_m(g,f_{\tilde w_n}), g\in Q_n\tilde w_n Q_n,$$
for $m$ large.
\item Let $f_{\tilde w_n}\in C_c^\infty(\Omega_{\tilde w_n},\omega)$ be as in Eq.$(\ref{eq4.1})$. For sufficiently large $m$ depending only on $f_{\tilde w_n}$ (and hence only on $f_,f_0$), we have 
$$B_m(\bt_n(a)\tilde w_n u_0, f_{\tilde w_n})=0,$$
for all $a\in \GL_n(F)$ and $u_0\in E_n^--(U_m\cap E^-_n) $, 
\item For a fixed $m$, and each integer $k$, the set $\wpair{a\in \GL_n(F)| B_m(\bt_n(a)\tilde w_n, f_{\tilde w_n} )\ne 0, |\det(a)|=q^k}$ is compact modulo $U_{\GL_n}$.
\end{enumerate}
\end{lem}
\noindent \textbf{Remark:} From (2) and (3) of the above lemma, one sees that 
$$B_m(\bt_n(a)\tilde w_n u_0,f)-B_m(\bt_n(a)\tilde w_n u_0,f_0)=0, $$
for $a\in \GL_n(F), u_0\in E_n^--(U_m\cap E_n^-)$ and $m$ large. This is certain ``stability" property of partial Bessel functions, which is the key in our proof of the local converse theorem. For similar ``stability" properties of partial Bessel functions for other groups and some special cases for $\Sp_{2r}$, see \cite[Lemma 6.2.2. and Lemma 6.2.6]{Ba95}, \cite[Proposition 5.7(c)]{Ba97}, \cite[Proposition 2.5 and Proposition 4.2]{Zh2} and \cite[Theorem 3.11]{Zh4}.\\

One can see in the following proof, we don't use the assumption $n<r$ and thus the assertions of the lemma work for $n=r$. 
\begin{proof}
(1) We have $\GL_r(F)=\cup_{w\in \bW_{\GL_r}} B_{\GL_r}wB_{\GL_n}\subset \cup_{w\in \bW_{\GL_r}}BwB$, where $B_{\GL_r}$ is the upper triangular Borel subgroup of $\GL_r(F)$, and $\bW_{\GL_r}$ is the Weyl group of $\GL_r$ which is viewed as a subgroup of $\bW$ by the embedding $\GL_r\cong M_r\subset G$. By Eq.(\ref{eq4.1}), it suffices to show that for $w\in \bW_{\GL_r} $ we cannot have $w\ge \tilde w_i$ for any $i\ge n$. In fact, if $w\ge \tilde w_i$, then $\tilde w_i\in \bW_{\GL_r}$. Contradiction. 

(2) Suppose that there exists $w\in \bW$ such that $C(w)\subset Q_n\tilde w_n Q_n \cap \Omega_{\tilde w_i} $ for some $i\ge n+1$. By Lemma \ref{lem2.7}, we have $$w_{\max}:= w_{\ell}^{L_n}\tilde w_n\ge w\ge \tilde w_i.$$
A matrix calculation shows that $w_{\ell} w_{\max}$ has the matrix form 
$$\begin{pmatrix} -I_n &&&\\ &&J_{r-n}&\\ &-J_{r-n}&& \\ &&&-I_n \end{pmatrix},$$
which is the long Weyl element of the Levi subgroup $M_{w_{\max}}\cong \GL_1^n \times \Sp_{2(r-n)}$. Thus $w_{\max}\in B(G)$. Moreover, the set of simple roots in $M_{w_{\max}}$ is 
$$\theta_{w_{\max}}=\wpair{\alpha_{i}, n+1\le i\le r }.$$
Since $\theta_{\tilde w_i}=\Delta-\wpair{\alpha_i}$, it follows that $\theta_{w_{\max}}\not\subset \theta_{\tilde w_i}$ for $i\ge n+1$. By Lemma \ref{lem2.6}, we cannot have $w_{\max}\ge \tilde w_i$. Contradiction. 

(3) Since $Q_n\tilde w_n Q_n$ is closed in $\Omega_{\tilde w_n}$ by Proposition \ref{prop2.5}(2), and $f_{\tilde w_n}\in C_c^\infty(\Omega_{\tilde w_n},\omega)$, the function $f_{\tilde w_n}$ has compact support on $Q_n \tilde w_n Q_n$. Thus there exists open compact subsets $Q'\subset Q_n, E'\subset E_n^-$ such that if $f_{\tilde w_n}( x\tilde w_n u)\ne 0$ for $x\in Q_n, u\in E_n^-$ implies that $x\in Q', u\in E'$. We take $m$ large enough such that $E'\subset U_m\cap E_n^-$. Note the choice of $m$ depends only on $E'$, which depends only on $f_{\tilde w_n}$. Recall that,  
$$B_m(\bt_n(a)\tilde w_n u_0,f_{\tilde w_n})=\frac{1}{\vol(U_m)}\int_{U\times U_m}\psi_U^{-1}(u')\psi_m^{-1}(u)f_{\tilde w_n}(u'\bt_n(a)\tilde w_n u_0 u) dudu',$$
see Eq.(\ref{eq2.2}). We write $u=u^+u^-$ with $u^+\in E^+_n\cap U_m, u^-\in E^-\cap U_m.$ By the previous lemma, $u_0'=(u^+)^{-1}u_0 u^+\in E^-_n-(E^-_n\cap U_m)$. We have 
\begin{align*}f_{\tilde w_n}(u'\bt_n(a)\tilde w_n u_0 u)&=f_{\tilde w_n}(u' \bt_n(a)\tilde w_n  u_0 u^+ u^-)\\
&=f_{\tilde w_n}(u' \bt_n(a)\tilde w_n  u^+ u_0'  u^-)\\
&=f_{\tilde w_n}(u' \bt_n(a) \tilde w_n u^+ \tilde w_n^{-1} \tilde w_n u_0'u^-).
\end{align*}
Note that $ u' \bt_n(a) \tilde w_n u^+ \tilde w_n^{-1}\in Q_n$. By the above analysis, if $f_{\tilde w_n}(u'\bt_n(a)\tilde w_n u_0 u)\ne 0$, we have $u' \bt_n(a) \tilde w_n u^+ \tilde w_n^{-1}\in Q' $ and $u_0'u^-\in E'\subset U_m\cap E_n^-$. Since $u^-\in U_m\cap E_n^-$, we get $u_0'\in U_m\cap E_n^-$. Contradiction. This implies that for any $u'\in U, u\in U_m$, 
$$ f_{\tilde w_n}(u'\bt_n(a)\tilde w_n u_0 u)=0.$$
Thus $$ B_m(\bt_n(a)\tilde w_n u_0,f_{\tilde w_n})=0.$$
(4) From the Iwasawa decomposition of $\GL_n$, it suffices to show that the set 
$$A(k)=\wpair{t=\diag(a_1,\dots,a_n)|a_i\in F^\times, \prod_{i}|a_i|=q^k, B_m(\bt_n(t)\tilde w_n, f_{\tilde w_n})\ne 0  }$$
is compact. Recall that, for a root $\beta$, we have fixed an isomorphism $\bx_\beta: F\ra U_\beta$, where $U_\beta$ is the root space of $\beta$.  For $i$ with $1\le i\le r$, we denote $\bx_{\alpha_i}$ by $\bx_i$ for simplicity. For example, if $r>2$, then $$\bx_2(x)=\bt_3\begin{pmatrix} 1&&\\ &1&x\\ &&1\end{pmatrix}.$$ 
For $i<n$, we have 
$$\tilde w_n \bx_{n-i}(x)=\bx_{i}(-x) \tilde w_n.$$
Take $t\in A(k)$, we have $$\bt_n(t)\tilde w_n\bx_{n-i}(x)=\bx_{i}(-\alpha_i(t) x) \bt_n(t)\tilde w_n. $$
For $x\in \fo$, we get $\bx_{n-i}(x)\in H_m$. By Eq.(\ref{eq2.1}), we get 
$$\psi(x)B_m(\bt_n(t)\tilde w_n,f_{\tilde w_n} )=\psi(-\alpha_i(t)x)B_m(\bt_n(t) \tilde w_n, f_{\tilde w_n}). $$
Since $ B_m(\bt_n(t)\tilde w_n,f_{\tilde w_n} )\ne 0$, we get 
$$ \psi(-\alpha_i(t)x)=\psi(x)=1, \forall x\in \fo.$$
Since $\psi$ has conductor $\fo$ by assumption, we get $\alpha_i(t)\in \fo$ for all $i$ with $1\le i<n$. This implies that $|a_1|\le |a_2|\le \dots \le |a_n|$. Similarly, we consider the root $\beta=\tilde w_n(\alpha_n)$, which is a negative root. We have $\tilde w_n \bx_\beta(x)=\bx_n(cx)\tilde w_n $, where $c\in \wpair{\pm 1}$. Thus we get 
$$\bt_n(t)\tilde w_n \bx_\beta(x)=\bx_n(\alpha_n(t)cx) t\tilde w_n. $$
Now take $x\in \fp^{(-2\rht(\beta)+1)m}$, so that $\bx_\beta(x)\in H_m $, see Lemma \ref{lem2.1}. By Eq.(\ref{eq2.1}), we have 
$$B_m(\bt_n(t) \tilde w_n, f_{\tilde w_n})=\psi(\alpha_n(t)cx)B_m(\bt_n(t)\tilde w_n,f_{\tilde w_n}). $$
Thus $\psi(\alpha_n(t)cx)=1 $ for all $x\in \fp^{(-2\rht(\beta)+1)m}$. Since $\psi$ has conductor $\fo$, we get $\alpha_n (t) \in \fp^{(2\rht(\beta)-1)m}$. From this condition, we get an upper bound of $|a_n|$. Now it is easy to see that $A(k)$ is compact.
\end{proof}

\begin{prop}\label{prop4.4}
Let $n$ be an integer with $1\le n<r$. The condition $\CC(n)$ implies that 
$$B_m(\bt_n(a)\tilde w_n, f_{\tilde w_n})=0,$$
for all $a\in \GL_n(F)$ and sufficiently large $m$ depending only on $f,f_0$.
\end{prop}
When $n=1$, the calculations appeared in the following proof was carried out in \cite{Zh4}. In the following proof, the integer $r$ and $n$ are fixed. For simplicity, we will drop $r,n$ from various sub- and sup-script. We will write $j(g)=w_{r-n,n}^{-1}g w_{r-n,n}$, for $g\in G$. Here we recall that $\bt_n(a)=\bm_r(a,1,\dots,1)=j(\bm_r(1,\dots,1,a))=j(\bm_n(a))$, and $\tilde w_n=j(w_n)$.
\begin{proof}
Let $m$ be a positive integer which is large enough such that all of the assertions of Lemma \ref{lem4.3} hold. We have defined a function $\phi_m=\phi_m^n\in \CS(F^n)$ in $\S$3.2. In the following proof, for simplicity, we will write the Weil representation as $\omega(g)\phi_m$ rather than $\omega_{\psi^{-1}}(g)\phi_m$. Note that the Weil representation is smooth, and thus for $i$ large enough, we have $\omega(s(\bar u))\phi_m=\phi_m$ for all $\bar u\in \ov{N}_{n,i}$. Here $s: K_{m_1}\ra \wt{\Sp}_{2n}(F)$ is a splitting of the double cover map $\wt{\Sp}_{2n}(F)\ra \Sp_{2n}(F)$. We also require that $m\ge m_1$.

Let $(\tau,V_\tau)$ be an irreducible generic representation of $\GL_n(F)$. Let $v\in V_\tau$. We consider an integer $i$ such that $i\ge \max\wpair{m, i_0(v), I(N'_{n,m}, v)}$ and such that $ \omega(s(\bar u))\phi_m=\phi_m$ for all $\bar u\in \ov{N}_{n,i}$. See $\S3.2$ for the notation $N'_{n,m}$ and $\S$3.1 for the definition of the notations $i_0(v)$ and $I(N'_{n,m}, v) $. By Lemma \ref{lem3.2}, we have a section $\xi_s^{i,v}\in \tilde V(s,\tau,\psi^{-1})$. 

We will compute the integral $\Psi(W_m^f, \phi_m, \xi_s^{i,v})$. Since $N_nM_n\ov{N}_n\subset \Sp_{2n}$ is open and dense, we will replace $U^n\setminus \Sp_{2n}$ by $U^n\setminus N_nM_n\ov{N}_n\cong U_{\GL_n}\setminus \GL_n\times \ov{N}_n $. Note that for $g=\bn_n(b)\bm_n(a) \bar u\in N_nM_n\ov{N}_n $, the quotient Haar measure on $ U^n\setminus N_nM_n\ov{N}_n\cong U_{\GL_n}\setminus \GL_n\times \ov{N}_n $ can be taken as $dg=|\det(a)|^{-(n+1)}d\bar u da.$
\begin{align*}&\Psi(W_m^f,\phi_m,\xi_s^{i,v})\\
&=\int_{U^n\setminus \Sp_{2n}}\int_{R^{r,n}}\int_{X_n}W_{m}^f (j(rxg)) \omega(g)\phi_m(x)\xi_s^{i,v}(g,I_n)dxdrdg\\
&=\int_{U_{\GL_n}\setminus \GL_n}\int_{\ov{N}_n} \int_{R^{r,n}}\int_{X_n} W_m^f(j(rx \bm_n(a) \bar u))\\ 
& \quad \quad \cdot\omega(\bm_n(a)\bar u ) \phi_m(x) \xi_s^{i,v}(\bm_n(a) \bar u, I_n)|\det(a)|^{-(n+1)}dx dr d\bar u  da.
\end{align*}
For $\bar u\notin \ov{N}_{n,i} $, we have $ \xi_s^{i,v}(\bm_n(a)s( \bar u), I_n)=0$. By assumption on $i$, we have $\omega(s(\bar u) ) \phi_m=\phi_m $ for $\bar u\in \ov{N}_{n,i}$. On the other hand, if we assume $\bar u=\bar \bn_n(b)\in \ov{N}_{n,i}$, we have 
$$j(\bar u)=w_{r-n,n}^{-1}\bar u w_{r-n,n}=\bar \bn_r\begin{pmatrix}0&0\\ b&0 \end{pmatrix}\in \ov{N}_r \cap H^r_m, $$
by the assumption $i\ge m$ and the definition of $\ov{N}_{n,i}$.

Thus by Eq.(\ref{eq2.1}), we have 
$$W_m^f(j(rx \bm_n(a) \bar u)  )=W_m^f(j(rx \bm_n(a) )).$$
From the above discussion, we get 
\begin{align}
\label{eq4.2}&\Psi(W_m^f,\phi_m,\xi_s^{i,v})\\
&=\vol(\ov{N}_{n,i})\int_{U_{\GL_n}\setminus \GL_n}\int_{R^{r,n}}\int_{X_n} W_m^f(j(rx \bm_n(a) )) \nonumber \\ 
& \quad \quad \cdot\omega(\bm_n(a) ) \phi_m(x) \xi_s^{i,v}(\bm_n(a) , I_n)|\det(a)|^{-(n+1)}dx dr   da. \nonumber
%&= \vol(\ov{N}_{n,i})\int_{U_{\GL_n}\setminus \GL_n}\int_{R^{r,n}}\int_{X_n} W_m^f(j(rx \bm_n(a) ))\\ 
%&\quad \quad \phi_m(xa) W_v^{\GL_n}(a) |\det(a)|^{s-[~]}dxdrda.
\end{align}
We can write  $$ r=\bm_r\begin{pmatrix} I_{r-n-1} &&y\\ &1&0\\ &&I_n \end{pmatrix}, rx=\bm_r\begin{pmatrix} I_{r-n-1} &&y\\ &1&x\\ &&I_n \end{pmatrix},$$ with $y\in \Mat_{(r-n-1)\times n}(F), x\in \Mat_{1\times n}(F).$
By abuse of notation, we will write 
$$r(y,x)=\bm_r\begin{pmatrix} I_{r-n-1} &&y\\ &1&x\\ &&I_n \end{pmatrix}, y\in \Mat_{(r-n-1)\times n}(F), x\in \Mat_{1\times n}(F).$$
%We can check that
%$$r(y,x) \bm_n(a)=\bm_n(a)r(ya,xa). $$

%After changing variables on $x,y$, we get 
%\begin{align*}
%&\Psi(W_m^f,\phi_m,\psi)\\
%&= \vol(\ov{N}_{n,i})\int_{U_{\GL_n}\setminus \GL_n}\int_{R^{r,n}}\int_{X_n} W_m^f\left(j\left(\bm_n(a) r(y,x)\right)\right)\\ 
%&\quad \quad \phi_m(x) W_v^{\GL_n}(a) |\det(a)|^{s-[~]}dxdyda.
%\end{align*}
%It is easy to see in the above integral, the domain of the integrand function has compact support with respect to $x,y$. 
From the matrix form, we can check that $$ j\left(r(y,x) \bm_n(a)\right)\in M_r.$$
By Lemma \ref{lem4.3}(1), we get 
\begin{align*}W_m^f\left(j\left( r(y,x)\bm_n(a)\right)\right)
=W_m^{f_0}\left(j\left( r(y,x)\bm_n(a)\right)\right).
\end{align*}
Thus by Eq.(\ref{eq4.2}), we can get 
$$\Psi(W_m^f,\phi_m, \xi_s^{i,v})=\Psi(W_m^{f_0},\phi_m, \xi_s^{i,v}).$$
By the local functional equation and the assumption $\gamma(s,\pi\times \tau,\psi)=\gamma(s,\pi_0\times \tau,\psi)$, we then get 
\begin{equation} \label{eq4.3}\Psi(W_m^f,\phi_m, \tilde \xi_{1-s}^{i,v})=\Psi(W_m^{f_0},\phi_m, \tilde \xi_{1-s}^{i,v}). \end{equation}
We now consider the integral $\Psi(W_m^f,\phi_m, \tilde \xi_{1-s}^{i,v}) $. Since $N_nM_nw_nN_n\subset \Sp_{2n}$ is open and dense, we will replace $U^n\setminus \Sp_{2n}$ by $U^n\setminus N_n M_nw_n N_n=U_{\GL_n}\setminus \GL_n w_n N_n$ in the integral of $\Psi(W_m^f, \phi_m,\tilde \xi_{1-s}^{i,v})$. We then have 
\begin{align*}
&\Psi(W_m^f, \phi_m, \tilde \xi_{1-s}^{i,v})\\
=&\int_{U_{\GL_n}\setminus \GL_n}\int_{N_n} \int_{R^{r,n}}\int_{X_n}W_m^f(j(r(y,x) \bm_n(a)w_n u ))\\
 &\cdot \omega(\bm_n(a)w_n u)\phi_m(x)\tilde \xi_{1-s}^{i,v}(\bm_n(a)w_n u, I_n)|\det(a)|^{-(n+1)}dxdyduda.
\end{align*}
There is a similar expression for $\Psi(W_m^{f_0},\phi_m, \xi_s^{i,v})$

We have 
\begin{align*}
j(r(y,x)\bm_n(a)w_n u )&=j(\bm_n(a) r(ya,xa)w_n u )\\
&=j(\bm_n(a) w_n r'(ya,xa) w_n u)\\
&=\bt_n(a) \tilde w_n j(r'(ya,xa))j(u),
\end{align*}
where $r'(y,x)=w_n^{-1}r(y,x)w_n$.  In matrix form, we have
\begin{align*}
r'(y,x)=\bn_r\begin{pmatrix}y&&\\x&&\\ &J_n{}^t\! x& J_n{}^t\! yJ_{r-n-1}  \end{pmatrix},
\end{align*}
and $$j(r'(y,x))=\bn_r\begin{pmatrix}J_n{}^t\! x& J_n{}^t\! y J_{r-n-1} & \\ &&y\\ &&x \end{pmatrix}.$$
For $u\in N_n$, we write $u=\bn_n(b)$. Then 
$$j(u)=\bn_r\begin{pmatrix}0&b\\ 0&0 \end{pmatrix},$$
and 
$$j(r'(y,x))j(u)=\bn_r\begin{pmatrix}J_n{}^t\! x& J_n{}^t\! y J_{r-n-1} &b \\ &&y\\ &&x \end{pmatrix}. $$
Thus $j(r'(y,x))j(u) \in E_n^-$. In particular, we have $ (j(r(y,x) \bm_n(a)w_n u )\in Q_n\tilde w_n Q_n$. By Lemma \ref{lem4.3}(2) and the above analysis, we have 
\begin{align*}&\quad W_m^f(j(r(y,x) \bm_n(a)w_n u ))-W_m^{f_0}(j(r(y,x) \bm_n(a)w_n u ))\\
&=B_m(j(r(y,x) \bm_n(a)w_n u ),f_{\tilde w_n} )\\
&=B_m(\bt_n(a)\tilde w_n j(r'(ya,xa)) j(\bn_n(b)),f_{\tilde w_n} )
\end{align*}
where we assumed $u=\bn_n(b)$.

From the above analysis and Eq.(\ref{eq4.3}), we get 
\begin{align}
0 =&\int_{U_{\GL_n}\setminus \GL_n}\int_{N_n} \int_{R^{r,n}}\int_{X_n}B_m(\bt_n(a)\tilde w_n j(r'(ya,xa)) j(\bn_n(b)),f_{\tilde w_n})  \label{eq4.4}\\
 &\cdot \omega(\bm_n(a)w_n \bn_n(b))\phi_m(x)\tilde \xi_{1-s}^{i,v}(\bm_n(a)w_n \bn_n(b), I_n)|\det(a)|^{-(n+1)}dydxdbda \nonumber \\
= &\int_{U_{\GL_n}\setminus \GL_n}\int_{N_n} \int_{R^{r,n}}\int_{X_n}B_m(\bt_n(a)\tilde w_n j(r'(ya,xa)) j(\bn_n(b)),f_{\tilde w_n}) \nonumber \\
 &\cdot \gamma_{\psi^{-1}}(\det(a))\omega(w_n \bn_n(b))\phi_m(xa)\tilde \xi_{1-s}^{i,v}(\bm_n(a)w_n \bn_n(b), I_n) |\det(a)|^{-(n+1)+1/2}dydxdbda \nonumber\\
=&\int_{U_{\GL_n}\setminus \GL_n}\int_{N_n} \int_{R^{r,n}}\int_{X_n}B_m(\bt_n(a)\tilde w_n j(r'(y,x)) j(\bn_n(b)),f_{\tilde w_n}) \nonumber \\
 &\cdot \gamma_{\psi^{-1}}(\det(a)) \omega(w_n \bn_n(b))\phi_m(x)\tilde \xi_{1-s}^{i,v}(\bm_n(a)w_n \bn_n(b), I_n) |\det(a)|^{-r-1/2}dydxdbda. \nonumber
\end{align}

Denote 
\begin{align*}D_m=\wpair{(x,y,b)|  j(r'(y,x))j(\bn_n(b))\in E_n^-\cap H_m }.
\end{align*}
By Lemma \ref{lem4.3}(3), if $(x,y,b)\notin D_m$, we have 
$$B_m(\bt_n(a)\tilde w_n j(r'(y,x)) j(\bn_n(b)),f_{\tilde w_n}) =0. $$
If $(x,y,b)\in D_m$, by Eq.(\ref{eq2.1}), we have 
$$ B_m(\bt_n(a)\tilde w_n j(r'(y,x)) j(\bn_n(b)),f_{\tilde w_n})=B_m(\bt_n(a)\tilde w_n ,f_{\tilde w_n}).$$
On the other hand, for $(x,y,b)\in D_m$, by Lemma \ref{lem3.4}, we have 
$$\omega(w_n \bn_n(b))\phi_m(x)=\omega(w_n )\phi_m(x).$$
By Eq.(\ref{eq3.2}) and the assumption on $i$, we have
$$\tilde \xi_{1-s}^{i,v}(\bm_n(a)w_n \bn_n(b), I_n)=\vol(\ov{N}_{n,i}) \gamma_{\psi^{-1}}^{-1}(\det(a))|\det(a)|^{(n+2)/2-s}W^*_v(a) .$$
A simple calculation shows that $\int_{x\in D_m\cap X_n}\omega(w_m)\phi_m(x)dx\ne 0$.
Now Eq.(\ref{eq4.4}) reads
$$0=\int_{U_{\GL_n}\setminus \GL_n}B_m(\bt_n(a)\tilde w_n, f_{\tilde w_n})\tilde W_v^{*}(a)|\det(a)|^{-s+(n+1)/2-r}da.  $$
This is true for any $\tau$ and $v\in V_\tau$.
Then by Proposition \ref{prop3.6} and Lemma \ref{lem4.3}(4), we get $B_m(\bt_n(a)\tilde w_n, f_{\tilde w_n})=0$ for all $a\in \GL_n(F)$.
\end{proof}

\begin{proof}[Proof of Proposition $\ref{prop4.1}$]
Let $w_{\max}=w_{\ell}^{L_n}\tilde w_n $ be as in the proof of Lemma \ref{lem4.3}. Then $M_{w_{\max}}\cong \GL_1^n\times \Sp_{2r-2n}$. Thus the center of $M_{w_{\max}}$ is $$A_{w_{\max}}=\wpair{\bm_r(\diag(a_1,a_2,\dots,a_n,1,\dots,1)), a_i\in F^\times, 1\le i\le n}\times Z.$$
Recall that $Z=\wpair{\pm I_{2r}}$ is the center of $G=\Sp_{2r}(F)$. For any $a\in A_{w_{\max}}$, we can write $a=z\bt_n(a_0)$ with $a_0$ a diagonal element in $\GL_n$. Since $\bt_n(a_0) w_{\ell}^{L_n}$ has the form $\bt_n( b)$ for certain $b\in \GL_n(F)$, we get  
$$B_m(aw_{\max}, f_{\tilde w_n})= B_m(aw_{\ell}^{L_n}\tilde w_n, f_{\tilde w_n})=\omega(z) B_m(\bt_n(a_0)w_{\ell}^{L_n}\tilde w_n, f_{\tilde w_n})=0$$
by Proposition \ref{prop4.4}.

 For any $w\in B(G)$ with $\tilde w_n\le w\le w_{\max}$, we can write $w=w'\tilde w_n$ for a Weyl element $w'$ of the Levi $L_n$, which has a realization of the form $\bt_n(b')$ for certain $b'\in \GL_n(F)$. We have $A_w\subset A_{w_{\max}}$ by Lemma \ref{lem2.6}. A similar consideration as above shows that 
 $$B_m(a w,f_{\tilde w_n})=0, \forall a\in A_{w}.$$
 
Denote 
$$\Omega'_{\tilde w_n}=\bigcup_{w \in B(G), w>w_{\max}\atop d_B(w, w_{\max})=1}\Omega_w.$$

By Theorem \ref{thm2.8}, there exists a function $f'_{\tilde w_n}\in C_c^\infty(\Omega'_{\tilde w_n}, \omega)$ such that 
$$B_m(g,f_{\tilde w_n})=B_m(g, f'_{\tilde w_n}),$$
for $m$ large enough. 
 
By a partition of unity argument, for each $w\in B(G)$ with $w>w_{\max}$ and $d_B(w,w_{\max})=1$, there exists a function $f'_{\tilde w_n, w}$ such that 
$$f'_{\tilde w_n}=\sum_{w\in B(G), w>w_{\max} \atop d_B(w,w_{\max})=1} f'_{\tilde w_n, w}.$$
Thus 
$$B_m(g,f_{\tilde w_n})=\sum_{w\in B(G), w>w_{\max} \atop d_B(w,w_{\max})=1} B_m(g, f'_{\tilde w_n, w}),$$
for $m$ large enough. From the proof Lemma \ref{lem4.3}, we see that 
$$\theta_{w_{\max}}=\wpair{\alpha_i, n+1\le i\le r}.$$
For $w\in B(G), w>w_{\max}$ and $d_B(w,w_{\max})=1$, by Lemma \ref{lem2.6}, we get 
$$\theta_w=\theta_{w_{\max}}-\wpair{\alpha_i},$$
for some $i$ with $n+1\le i\le r$. Recall that 
$$\theta_{\tilde w_i}=\Delta-\wpair{\alpha_i}.$$
Thus we have $ \theta_w\subset \theta_{\tilde w_i}$, and hence $w\ge \tilde w_i$. Thus $\Omega_{w}\subset \Omega_{\tilde w_i}$. By Proposition \ref{prop2.5}, $\Omega_w$ is in fact open in $\Omega_{\tilde w_i}$. Thus by the basic exact sequence Eq.(\ref{eq2.3}), $f_{\tilde w_n,w}\in C_c^\infty(\Omega_w,\omega)\subset C_c^\infty(\Omega_{\tilde w_i}, \omega)$. Now define 
$$f'_{\tilde w_i}=f'_{\tilde w_n,w}+f_{\tilde w_i}\in C_c^\infty(\Omega_{\tilde w_i},\omega). $$
By the above argument and Eq.(\ref{eq4.2}), we get 
$$B_m(g,f)-B_m(g,f_0)=\sum_{i=n+1}^r B_m(g,f'_{\tilde w_i}).$$
This finishes the proof of Proposition \ref{prop4.1}.
\end{proof}

\section{Proof of the main theorem}\label{sec7}
In this section, we finish the proof of Theorem \ref{thm2.10}.

By Proposition \ref{prop4.1}, the condition $\CC(r-1)$ implies that there exists a function $f_{\tilde w_r}\in C_c^\infty(\Omega_{\tilde w_r},\omega)$
\begin{equation}\label{eq5.1}B_m(g,f)-B_m(g,f_0)=B_m(g,f_{\tilde w_r}),\end{equation}
for all $g\in G$ and large enough $m$ depending on $f,f_0$. Note that $\tilde w_r=w_r$.

Recall that, in $\S$3.3, we have constructed a Schwartz function $\phi_k=\phi_{k,r}^r\in \CS(F^r)$ for $k>0$.

\begin{prop}\label{prop5.2}
Let $m$ be a large enough integer and $k$ be an integer such that $k\ge m$. Assume the condition $\CC(r)$, we have 
$$B_m(\bm_r(a)\tilde w_r,f_{\tilde w_n}) (\omega(w_r)\phi_k)(e_ra)=0, $$
for all $a\in \GL_r(F)$.
\end{prop}
\begin{proof}
The proof is similar to that of Proposition \ref{prop4.4}. In the following proof, we still write the Weil representation $\omega_{\psi^{-1}}$ as $\omega$ for simplicity.

We take the integer $m$ large such that Lemma \ref{lem4.3}(3) and Eq.(\ref{eq5.1}) hold. We then take a positive integer $k$ such that $k\ge m$.  Let $(\tau,V_\tau)$ be an irreducible generic representation of $\GL_r(F)$. Fix a vector $v\in V_\tau$. Let $i$ be a positive integer with $i\ge\max\wpair{m,i_0(v),I(N_{r}\cap U_m,v)}$ and such that $\omega(s(\bar u))\phi_k=\phi_k$ for all $\bar u\in \ov{N}_{r,i}$. We then have a section $\xi_s^{i,v}\in \tilde V(s,\tau,\psi^{-1})$.

As in the proof of Proposition \ref{prop4.1},  we have 
$$\Psi(W_m^f,\phi_k, \xi_s^{i,v})=\Psi(W_m^{f_0},\phi_k,\xi_s^{i,v}).$$
Thus from the assumption on the local gamma factors and the local functional equation, we can get 
\begin{equation}\label{eq5.2}\Psi(W_m^f,\phi_k, \tilde\xi_{1-s}^{i,v})=\Psi(W_m^{f_0},\phi_k,\tilde\xi_{1-s}^{i,v}).\end{equation}
Similar as in the proof of Proposition \ref{prop4.1}, we have
\begin{align*}
&\Psi(W_m^f,\phi_k,\tilde \xi_{1-s}^{i,v})\\
=&\int_{U_{\GL_r}\setminus \GL_r}\int_{N_r}W_m^{f}(\bm_r(a)w_r u)\omega(\bm_r(a)w_r u)\phi_k(e_r)\\
\quad \cdot &\tilde \xi_{1-s}^{i,v}(\bm_r(a)w_ru,I_r)|\det(a)|^{-(r+1)}duda.
\end{align*}
There is a similar expression for $\Psi(W_m^{f_0},\phi_k,\tilde \xi_{1-s}^{i,v})$. By Eq.(\ref{eq5.1}) and Eq.(\ref{eq5.2}), we can get 
\begin{align}
0&\equiv \int_{U_{\GL_r}\setminus \GL_r}\int_{N_r}B_m(\bm_r(a)w_r u, f_{\tilde w_r}) \omega(\bm_r(a)w_r u)\phi_k(e_r) \label{eq5.3}\\
&\quad \cdot \tilde \xi_{1-s}^{i,v}(\bm_r(a)w_r u,I_r)|\det(a)|^{-(r+1)}duda. \nonumber
\end{align}
By Lemma \ref{lem4.3}(3), if $u\notin U_m\cap N_r$, we then get $$B_m(\bm_r(a)w_r u, f_{\tilde w_r})=0. $$
For $u\in U_m\cap N_r$, by Eq.(\ref{eq2.1}), we have 
$$B_m(\bm_r(a)w_r u, f_{\tilde w_r})=\psi_U(u)B_m(\bm_r(a)w_r , f_{\tilde w_r}).$$
By Lemma \ref{lem3.5}(1), for $u\in N_r\cap U_m\subset N_r\cap U_k$, we have 
$$\omega(u)\phi_k=\psi_U^{-1}(u)\phi_k. $$
On the other hand, by Eq.(\ref{eq3.2}) and our assumption on $i$, for $u\in N_r\cap U_m$, we have
\begin{align*}&\omega(\bm_r(a)w_r u)\phi_k(e_r) \tilde \xi_{1-s}^{i,v}(\bm_r(a)w_n u)\\
=&\vol(\ov{N}_{r,i})|\det(a)|^{(r+3)/2-s}\omega(w_r)\phi_k(e_ra)W^*_v(a).
\end{align*}
Now Eq.(\ref{eq5.3}) becomes 
$$\int_{U_{\GL_r}\setminus \GL_r} B_m(\bm_r(a)w_r,f_{\tilde w_r})\omega(w_r)\phi_k(e_ra)W_v^*(a) |\det(a)|^{-s+(-r+1)/2}da=0.$$
The above equation is true for all irreducible generic representations $(\tau,V_\tau)$ and all $v\in V_\tau$. Then by Proposition \ref{prop3.6} and Lemma \ref{lem4.3}(4), we get $$B_m(\bm_r(a)w_r,f_{\tilde w_r})\omega(w_r)\phi_k(e_ra)=0$$ for all $a\in \GL_r(F)$.
\end{proof}

We now can finish the proof of Theorem \ref{thm2.10}.
\begin{proof}[Proof of Theorem $\ref{thm2.10}$]
We are going to show that the condition $\CC(r)$ implies that $B_m(g,f)=B_m(g,f_0)$ for sufficiently large $m$ depending only on $f,f_0$. By Eq.(\ref{eq5.1}), it suffices to show that $B_m(g,f_{\tilde w_r})=0$ for all $g\in G$. By Theorem \ref{thm2.8}, it suffices to show that $B_m(aw,f_{\tilde w_r})=0$ for all $a\in A_w$ and all $w\in B(G)$ with $w\ge \tilde w_r=w_r$.

By Lemma \ref{lem2.6}, for $w\in B(G), w\ge w_r$, we have $\theta_w\subset \theta_{w_r}=\Delta-\wpair{\alpha_r}$. In particular, $\beta:=-w(\alpha_r)>0$. Let $t=\bm_r(\diag(a_1,\dots,a_r))\in A$ and $x\in \fp^{(2\rht(\beta)+1)m}$. We have $\bx_{-\beta}(x)\in U_{-\gamma,m}$ by Lemma \ref{lem2.1}. On the other hand, we have 
$$tw\bx_{-\beta}(x)=t\bx_{\alpha_r}(cx) w=\bx_{\alpha_r}(a_r^2 cx)tw,$$
where $c\in \wpair{\pm 1}$. Thus by Eq.(\ref{eq2.1}), we have 
$$B_m(tw,f_{\tilde w_r})=\psi(a_r^2 cx)B_m(tw,f_{\tilde w_r}). $$
Thus if $B_m(tw,f_{\tilde w_r})\ne 0$, we get $\psi(a_r^2 cx)=1$ for all $x\in \fp^{(2\rht(\beta)+1)m}$ and thus $a_r^2\in \fp^{-(2\rht(\beta)+1)m}$. Equivalently, if $a_r^2\notin \fp^{-(2\rht(\beta)+1)m}$, we have $B_m(tw,f_{\tilde w_r})=0$.

On the other hand, we can write $w=w'w_n$ with $w'\in \bW_{\GL_r}$. Here $\bW_{\GL_r}$ is the Weyl group of $\GL_r$ and any $w'\in \bW_{\GL_r}$ has a representative of the form $\bm_r(a_0)$ with $a_0\in \GL_r(F)$ a permutation matrix. Let $a=\diag(a_1,\dots,a_r) a_0$. Thus we can write $tw=\bm_r(a)w_n$. 

We now assume that $a_r^2\in \fp^{-(2\rht(\beta)+1)m}$. Note that $a_r$ is in the last row of the matrix $a$. By Lemma \ref{lem3.5}(2), we can take $k$ large such that $\omega(w_r)\phi_k(e_ra)\ne 0$. Then by Proposition \ref{prop5.2}, we get 
$$B_m(tw,f_{\tilde w_r})=B_m(\bm_r(a)w_r, f_{\tilde w_r})=0.$$
This finishes the proof of Theorem \ref{thm2.10}. 
\end{proof}
Note that, throughout the above proof, we fixed an unramified character $\psi$. If $\pi$ and $\pi_0$ are  generic with respect to $\psi_U$ for a ramified character $\psi$, we only need to modify the definition of Howe vectors and similar proofs go through.

\section{Local converse theorem for unitary groups}\label{sec8}

\subsection{The local converse theorem for $\RU(r,r)$}

Let $E/F$ be a quadratic extension of local fields and let $\RU(r,r)$ be the unitary group of rank defined by the skew-Hermitian form $\begin{pmatrix}&J_r\\ -J_r& \end{pmatrix}$. In \cite{BAS}, global and local zeta integrals for $\RU(r,r)\times \GL_n(E)$ for $1\le n\le r$ were studied. The local integrals are quite similar to that of the symplectic groups which were reviewed in Section \ref{sec3}. We give a very brief review of that. Let $\mu$ be a fixed character of $E^\times$ such that $\mu|_{F^\times}$ is the class field theory character. The skew Hermitian form defines an embedding $\RU(r,r)\incl \Sp_{4r}(F)$.  The character $\mu$ defines a splitting of the double cover $\wt{\Sp}_{4r}(F)\ra \Sp_{4r}(F)$ over $\RU(r,r)$, i.e., defines an embedding $\RU(r,r)\ra \wt{\Sp}_{4r}(F)$. Thus for a nontrivial additive character $\psi$ of $F$, one has a Weil representation $\omega_{\mu,\psi}$ of $\RU(r,r)$ on $\CS(E^r)$.

Note that the Levi subgroup of the Siegel parabolic subgroup $P_r$ of $\RU(r,r)$ is isomorphic to $\GL_r(E)$. Given a generic irreducible representation $\tau$ of $\GL_r(E)$, we can consider the induced representation $I(s,\tau)=\Ind_{P_r}^{\RU(r,r)}(\tau\otimes |\det|_E^{s-1/2})$. Fix a Whittaker functional $\lambda$ of $\tau$. For $f_s\in I(s,\tau)$, we can consider the function $\xi_s(g,a)=\lambda(\tau(a)f_s(g))$ on $\RU(r,r)\times \GL_r(E)$. Let $V(s,\tau)$ be the space of $\xi_s$.

Let $U=U^r$ be the upper triangular unipotent subgroup of $\RU(r,r)$ and let $\psi_U$ be a generic character of $U$. Let $\pi$ be an irreducible $\psi_U$-generic representation of $\RU(r,r)$ and $\tau$ be a generic irreducible smooth representation of $\GL_n(E)$ with $1\le n\le r$. For $W\in \CW(\pi,\psi_U)$, $\xi_s\in V(s,\tau)$ and $\phi\in \CS(E^r)$, the Shimura type integral for $\pi\times \tau$ is defined by 
\begin{align}
&\Psi(W,\phi,\xi_s)\\
=&\left\{\begin{array}{lll}\int_{U^n\setminus \RU(n,n)}\int_{R^{r,n}}\int_{X_n} W(w_{r-n,n}^{-1}(rxg)w_{r-n,n})\omega_{\mu,\psi^{-1}}(g)\phi(x)\xi_s(g,I_n)dxdrdg, & n<r,\\
                                                              \int_{U^r\setminus \RU(r,r)} W(g)\omega_{\mu,\psi^{-1}}(g)\phi(e_r)\xi_s(g,I_r)dg, & n=r.        \end{array}\right. \nonumber\end{align}

Here $R^{r,n}, X_n, w_{r-n,n}$ can be defined similarly as in the symplectic group case.

Following the uniqueness of Fourier-Jacobi model in the unitary group case \cite{GaGP,Su}, one can define local gamma factors via local functional equations: there exists a meromorphic function $\gamma(s,\pi\times (\tau \mu),\psi)$ such that 
$$\Psi(W,\phi,M(s,\tau)\xi_s)=\gamma(s,\pi\times (\tau \mu),\psi )\Psi(W,\phi,\xi_s),$$
for all $W\in \CW(\pi,\psi_U),\phi\in \CS(E^n),\xi_s\in V(s,\tau)$. Here $M(s,\tau)$ is the standard intertwining operator. 

The same proof as in the symplectic case will give the local converse theorem in the unitary group case. We just record the theorem here. 
\begin{thm}
Let $\pi,\pi_0$ be two $\psi_U$-generic irreducible smooth representations of $\RU(r,r)$ with the same central character. If $\gamma(s,\pi\times(\tau \mu),\psi)=\gamma(s,\pi_0\times (\tau\mu) ,\psi)$ for all irreducible generic representations of $\GL_n(E)$ with $1\le n\le r$, then $\pi\cong \pi_0$.
\end{thm}
Some small ranked case of the above theorem has been worked out in \cite{Zh1,Zh2}. Recently, K. Morimoto proved the above theorem using descent theory, see \cite[Theorem 9.4]{M}. It should be noted that the local gamma factors Morimoto used are arising from Langlands-Shahidi method. It is still not known that the local gamma factors we used are essentially the same with Langlands-Shahidi gamma factors, although it is expected to be true like the $\Sp_{2n}$ case.

\subsection{Unitary group at split places} 
In this final subsection, we explain how our method might give a new proof of the Jacquet's local converse conjecture for $\GL_{2n}$.

Let $E/F$ be a quadratic extension of number fields and let $\RU(r,r)$ be the unitary group associated with $E/F$ defined by the skew-Hermitian form as in last subsection. In \cite{BAS}, the authors considered a global zeta integral associated with a generic cusp automorphic representation $\pi$ of $\RU(r,r)(\BA_F)$, a generic cuspidal automorphic representation $\tau$ of $\GL_n(\BA_E)$ and a Bruhat-Schwartz function on the space of $\BA_E^n$. This integral is Eulerian, represent the local L-function of $\RU(r,r)(F_v)\times \GL_n(E_v)$ at unramified places, and give the local zeta integral considered in last subsection at inert places. If $v$ is a place of $F$ which splits over $E$, then $E_v=F_v^2$, $\RU(r,r)(F_v)=\GL_{2r}(F_v)$ and $\GL_n(E_v)=\GL_n(F_v)\oplus \GL_n(F_v)$. At such a place $v$, the representation $\tau_v$ is a pair $(\tau_{1,v},\tau_{2,v})$ of generic representations of $\GL_n(F_v)$.  Thus at a split place $v$, the local functional equation of the corresponding local zeta integral of \cite{BAS} gives a local gamma factor $\gamma(s, \pi_v,\tau_{1,v},\tau_{2,v},\psi_v)$. At the unramified places, one has $\gamma(s,\pi_v,\tau_{1,v},\tau_{2,v},\psi)=\gamma(s,\pi_v\times \tau_{1,v},\psi_v)\gamma(s,\tilde\pi_v\times \tau_{2,v},\psi)$ up to a normalizing factor, where $\tilde \pi_v$ is the contragredient representation of $\pi_v$, and $\gamma(s,\pi_v\times \tau_{1,v},\psi_v)$ is the standard local gamma factor for $\GL_{2r}\times \GL_n$. One would expect this is true in general. 

Now assume $F$ is a $p$-adic field and $\pi,\pi_0$ are two irreducible generic supercuspidal representation of $\GL_{2r}(F)$ with the same central character. Our method of the proof of the local converse theorem should also give a local converse theorem of $\GL_{2r}$ using the local gamma factors $\gamma(s,\pi,\tau_1,\tau_2,\psi) $, i.e., if $\gamma(s,\pi,\tau_1,\tau_2,\psi)=\gamma(s,\pi_0,\tau_1,\tau_2,\psi)$ for all pairs $(\tau_1,\tau_2)$ of irreducible generic representations of $\GL_n(F)$ with $1\le n\le r$, then we have $\pi\cong \pi_0$. This would give a different proof of Jacquet's local converse conjecture for $\GL_{2n}$ modulo the expected property $ \gamma(s,\pi,\tau_1,\tau_2,\psi)=\gamma(s,\pi\times \tau_1,\psi)\gamma(s,\tilde \pi\times \tau_2,\psi)$. To obtain a proof of the Jacquet's conjecture for $\GL_{2r+1}$, one should consider local gamma factors of the local zeta integrals at the split places of the unitary group $\RU_{2r+1}$ of \cite{BAS}. The author is working on these projects. This is in fact the original motivation of author's work on the local converse theorem for unitary group.

\end{document}